\documentclass[12pt]{amsart}
\usepackage[margin=1in]{geometry}
\usepackage{graphicx}
\usepackage{turnstile}
\usepackage{amssymb}
\usepackage{amsmath}
\usepackage{amsfonts}
\usepackage{amstext}
\usepackage{amsthm}
\usepackage{setspace}
\usepackage{array}
\usepackage{turnstile}
\usepackage{textcomp}
\usepackage{tipa}
\usepackage{esvect}
\usepackage{graphicx}

\numberwithin{equation}{section}

\newtheorem{Lem}{Lemma}
\newtheorem*{thm}{Theorem}
\newtheorem{Thm}{Theorem}

\newtheorem{definition}{Definition}
\newtheorem{remark}{Remark}
\newtheorem{corollary}{Corollary}

\newcommand{\cl}[1]{\ensuremath{\overline{#1}}}
\newcommand{\set}[1]{\ensuremath{\{#1\}}}

\newcommand{\E}[1]{\ensuremath{E\left[ #1 \right]}}

\newcommand{\curly}[1]{\ensuremath{\mathcal #1}}
\newcommand{\leb}{\ensuremath{\lambda}}
\newcommand{\ep}{\ensuremath{\epsilon}}

\newcommand{\C}{\ensuremath{\mathbb C}}
\newcommand{\w}{\ensuremath{\omega}}

\newcommand{\Di}[2]{\frac{\partial #2}{\partial #1}}

\newcommand{\DDi}[2]{\frac{\partial^2 #2}{\partial #1^2}}
\newcommand{\twiddle}[1]{\ensuremath{\widetilde{#1}}}

\newcommand{\case}[2]{\ensuremath{#1, \text{~if~} #2}}

\newcommand{\setst}[2]{\ensuremath{\left\{#1\,\middle|\,#2\right\}}}
\newcommand{\intersection}{\ensuremath{\cap}}

\newcommand{\abs}[1]{\left\lvert #1 \right\rvert}
\newcommand{\norm}[1]{\left\lVert#1\right\rVert}

\newcommand{\inprod}[2]{\ensuremath{\left\langle#1,#2\right\rangle}}

\newcommand{\gives}{\ensuremath{\rightarrow}}

\newcommand{\x}{\ensuremath{\times}}

\renewcommand{\Re}{\ensuremath{\mathrm{Re} \ }}

\DeclareMathOperator{\Cov}{Cov}
\DeclareMathOperator{\Sym}{Sym}

\title[Correlations and Nearest Neighbor Spacings]{Correlations and Pairing between Zeros and Critical Points of Gaussian Random Polynomials}            

\begin{document}
\author{Boris Hanin} 
\address{Department of Mathematics, Northwestern University, 2033 Sheridan Rd. Evanston, IL, 60208}
\email{bhanin@math.northwestern.edu}

\maketitle                    
\begin{abstract}
We study the asymptotics of correlations and nearest neighbor spacings between zeros and holomorphic critical points of $p_N,$ a degree $N$ Hermitian Gaussian random polynomial in the sense of Shiffman and Zeldtich, as $N$ goes to infinity. By holomorphic critical point we mean a solution to the equation $\frac{d}{dz}p_N(z)=0.$ Our principal result is an explicit asymptotic formula for the local scaling limit of $\E{Z_{p_N}\wedge C_{p_N}},$ the expected joint intensity of zeros and critical points, around any point on the Riemann sphere. Here $Z_{p_N}$ and $C_{p_N}$ are the currents of integration (i.e. counting measures) over the zeros and critical points of $p_N,$ respectively. We prove that correlations between zeros and critical points are short range, decaying like $e^{-N\abs{z-w}^2}.$ With $\abs{z-w}$ on the order of $N^{-1/2},$ however, $\E{Z_{p_N}\wedge C_{p_N}}(z,w)$ is sharply peaked near $z=w,$ causing zeros and critical points to appear in rigid pairs. We compute tight bounds on the expected distance and angular dependence between a critical point and its paired zero. 
\end{abstract}

\setcounter{section}{0}
\section{Introduction}
\noindent Let $p_N$ be a degree $N$ polynomial in one complex variable. We study in this paper how its zeros and holomorphic critical points (those $z$ for which $\frac{d}{dz}p_N(z)=0$) are correlated when $p_N$ is random and $N$ is large. To motivate the study of correlations between zeros and holomorphic critical points, we recall the following classical theorem from complex analysis. 
\begin{thm}[Gauss-Lucas]
  The holomorphic critical points of any polynomial in one complex variable are contained in the convex hull of its zeros.
\end{thm}
Non-trivial correlations between zeros and critical points of random polynomials must therefore always exist. We prove in this paper that, at least for Hermitian Gaussian random polynomials in the sense of Bleher, Shiffman, and Zelditch in \cite{PLL, Universality, NV} (cf Section \ref{S:Def} for a definition), a zero of $p_N$ at $z$ and a holomorphic critical point of $p_N$ at $w$ are essentially uncorrelated unless $\abs{z-w}$ is on the order of $N^{-1/2}.$ This follows from Theorems \ref{T:Crit Two Point Function} and \ref{T:Crit Given Zero One Point Function}. On the $N^{-1/2}$ length-scaled, however, we find that, on average, zeros and critical points appear in rigid pairs (cf Figures 1-3). This statement is quantified in Theorems \ref{T:Crit Two Point Function} and \ref{T:Sendov}. 

We assume from now on that $p_N$ is a degree $N$ Hermitian Gaussian random polynomial. We associate to $p_N$ the currents of integration (equivalently counting measures)
\[Z_{p_N}:=\sum_{p_N=0}\delta_z\quad \text{ and }\quad C_{p_N}:=\sum_{\frac{d}{dw}p_N(w)=0}\delta_w\]
over its zeros and holomorphic critical points and study the expected joint intensity 
\[K_N:=\E{Z_{p_N}\wedge C_{p_N}}.\]
We will refer to $K_N$ as the cross-correlation current. As in \cite{PLL, Universality, NV, Conditional} and elsewhere, our methods combine the Poincar\'e-Lelong formula with Sz\"ego kernel asymptotics in the setting of positive holomorphic line bundles over compact complex manifolds. 

This paper is the first to consider correlations and nearest neighbor spacings between zeros and holomorphic critical points of $p_N.$ Critical points with respect to smooth metric connections were considered in \cite{VacuaI, VacuaII, VacuaIII, Transportation} and, as explained in Section \ref{S:Smooth VS Holomorphic}, result in a significantly different theory. Perhaps the most striking difference is that zeros and holomorphic critical points are highly correlated and tend to appear in rigid pairs. This is illustrated in Figures 1-4 for $p_{50},$ a random degree $50$ polynomial drawn from the computationally tractable $SU(2)$ ensemble described in Section \ref{S:SU2}. The colored lines in these figures are the (negative) gradient flow lines of $M(z):=\abs{p_{50}(z)}^2.$ Zeros and holomorphic critical points of $p_{50}$ are the local minima and saddle points of $M,$ respectively. Flow lines terminating in a given zero or critical point are drawn in the same color. 
\begin{figure}[h]
\centering
\includegraphics[scale=.75]{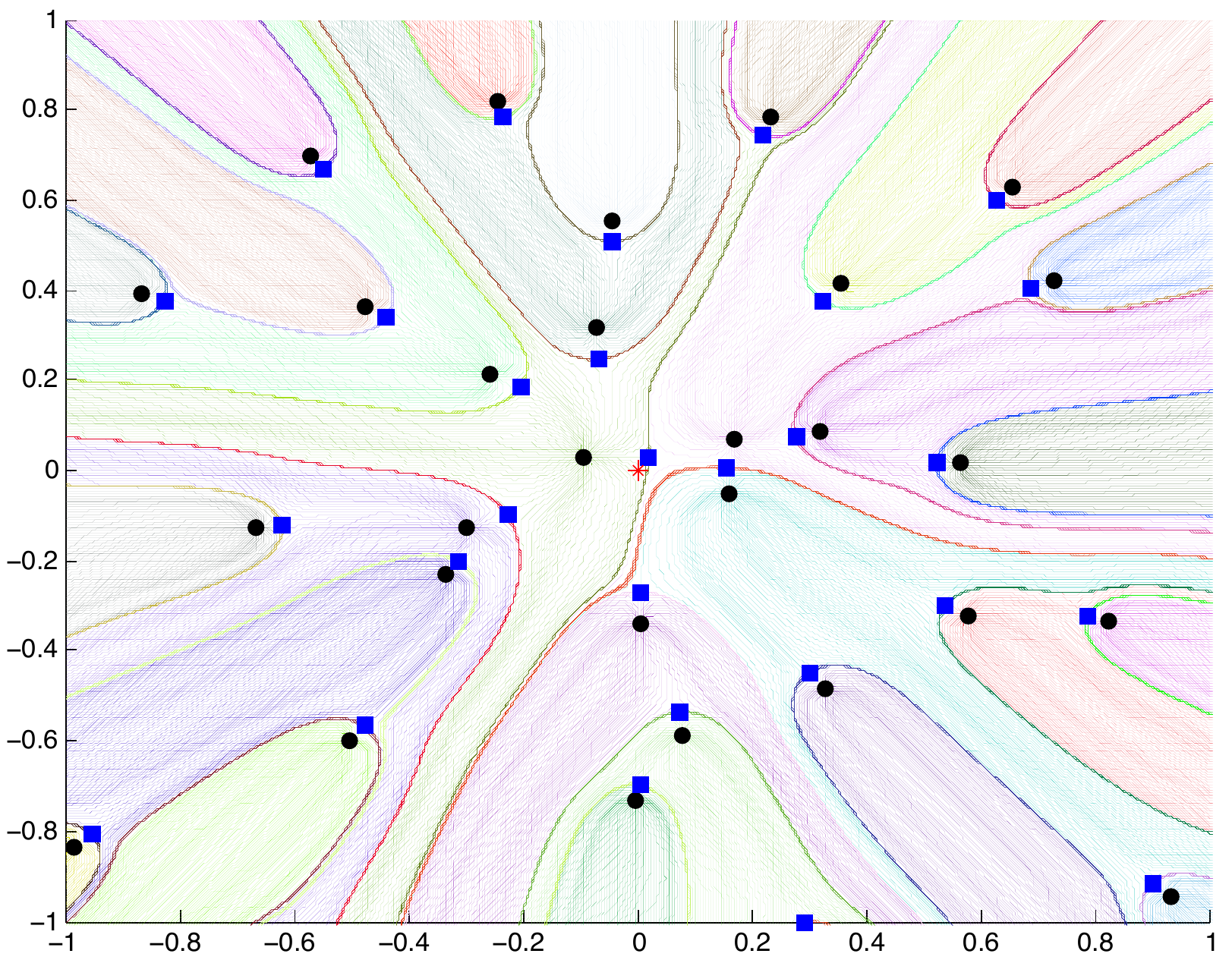}
\label{F:Unscaled SU2 50}
\caption{Zeros (black discs) and holomorphic critical points (blue squares) for an $SU(2)$ polynomial $p_{50}$ of degree $50.$ The origin is denoted by a red asterisk.}
\end{figure}

$M(z)$ is subharmonic and so cannot have any local maxima. The basin of attraction for a given zero (i.e. those points in $\C$ whose gradient flow lines terminate in that zero) is therefore unbounded. By comparison, the basins of attraction considered by Nazarov, Sodin, and Volberg in \cite{Transportation} are compact and have constant area with probability $1.$ The difference is that while they study the zeros of a Gaussian Analytic Function $f,$ the saddle points of their potential are critical points of the random smooth function $f(z) e^{-\frac{1}{2}\abs{z}^2}$ rather than $f$ itself. Their critical points are therefore computed with respect to the metric connection of the hermitian metric $h(z)=\frac{1}{\pi}e^{-\abs{z}^2}$ on the trivial line bundle $\C\x \C\twoheadrightarrow \C$. This paper investigates the purely holomorphic setting where, as explained in Section \ref{S:Smooth VS Holomorphic}, we do not include the metric factor $e^{-\frac{1}{2}\abs{z}^2}$ in contructing the potential. 
 
\subsection{Definitions and Notation}\label{S:Def}
Let $h$ be a smooth positive Hermitian metric on $\mathcal O(1)\twoheadrightarrow \C P^1.$ We recall the definition of the Hermitian Gaussian ensemble associated to $h.$ Fix $N\geq 1$ and write $\curly P_N$ for the space of polynomials of degree at most $N$ in one complex variable. We identify $\curly P_N$ with $H_{hol}^0(\C P^1, \mathcal O(N)),$ the space of global sections of $\mathcal O(N)\twoheadrightarrow \C P^1,$ by the linear map $z^j\mapsto z_1^j z_0^{N-j}$ (cf Section \ref{S:CP1}). A random polynomial (section) of degree $N$ drawn from this ensemble is
\[p_N:=\sum_{j=0}^N a_j S_j,\]
where $a_j\sim N(0,1)_{\C}$ are i.i.d. standard complex Gaussians and $\{S_j\}_{j=0}^N$ is any orthonormal basis for $H_{hol}^0(\C P^1,\mathcal O(N))$ with respect to the inner product
 \begin{equation}
  \label{E:Inner Product}
\inprod{s_1}{s_2}_h:=\int_{\C P^1}h^N(s_1(z),s_2(z))\, \w_h(z),\quad s_1,s_2\in H_{hol}^0(\C P^1,\mathcal O(N)).
\end{equation}
Here $\w_h:=\frac{i}{2\pi}\partial \cl{\partial} \log h^{-2}$ is the first Chern class of $(\mathcal O(1),h).$ We say that $p_N$ is a random polynomial of degree $N$ drawn from the Hermitian Gaussian ensemble corresponding to $h.$ 

We denote throughout by $z_0$ the usual frame of $\mathcal O(1)$ over $\C P^1\backslash \set{\infty}$ (cf Section \ref{S:CP1}) and define $\nabla^{z_0}$ to be the meromorphic connection on $\mathcal O(1)$ for which $z_0$ is parallel (cf Section \ref{S:Connection}). Writing $p_N$ for a degree $N$ polynomial and the section it represents, the critical point equation $\frac{d}{dz}p_N(z)=0$ becomes ${\nabla^{z_0}}^{\otimes N}p_N=0.$ Relative to the frame $z_0,$ $\nabla^{z_0}$ has an indentically zero connection $1-$form. Hence, it is related to the metric connection $\nabla^h$ of $h$ via
\begin{equation}\label{E:Two Connections}
\nabla^{z_0}=\nabla^h-\partial \phi_{z_0},  
\end{equation}
where
\begin{equation}
  \label{E:Special KPot}
\phi_{z_0}(z):=\log \norm{z_0(z)}_h^{-2}.  
\end{equation}
The function $\phi_{z_0}$ will play an important role in our results. 
\begin{figure}[h]
\centering
\includegraphics[scale=.75]{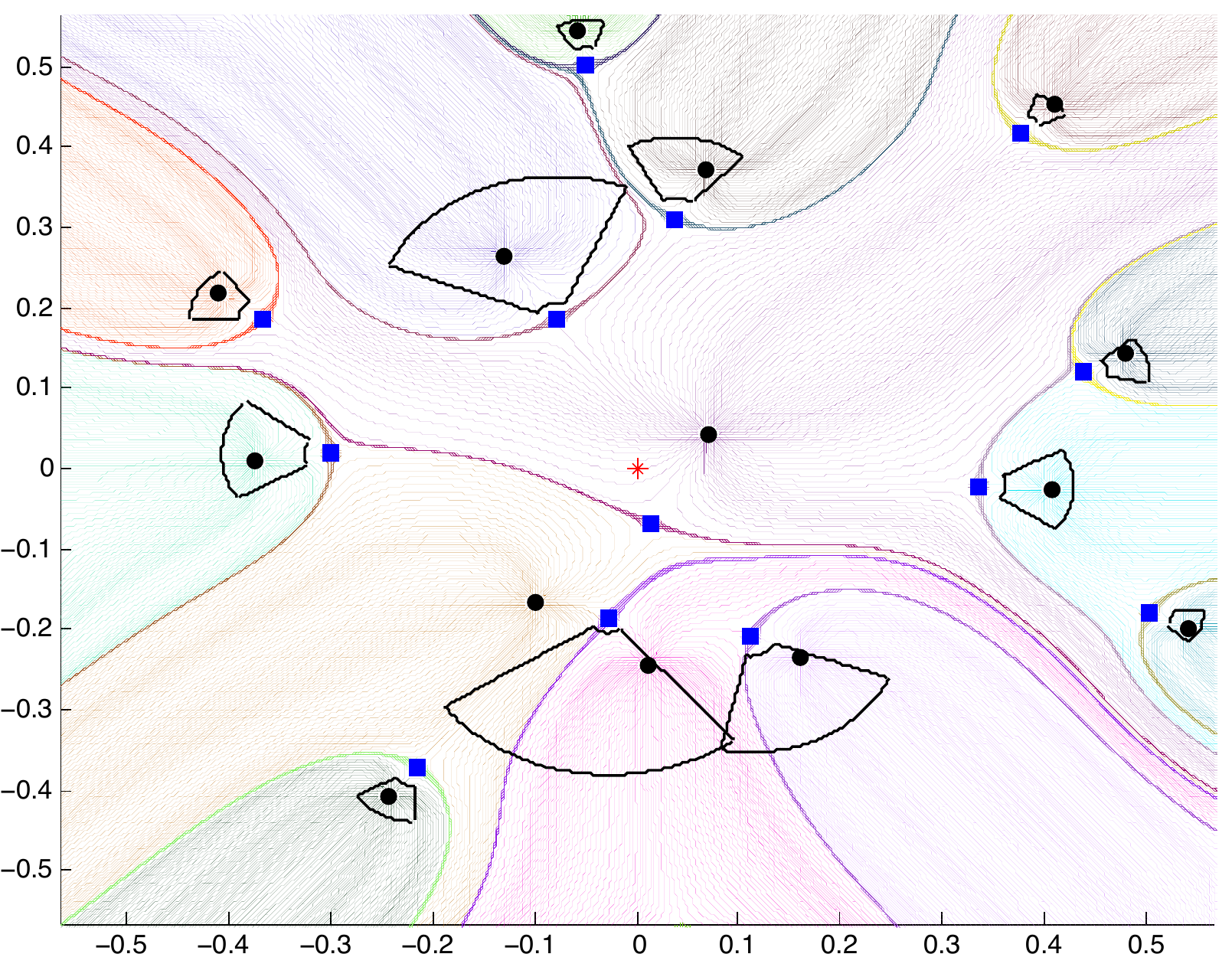}
\label{F:Scaled SU2 Balls Crit}
\caption{Zeros (black discs) and holomorphic critical points (blue squares) for an $SU(2)$ polynomial $p_{50}$ of degree $50$ inside $[-4/\sqrt{50},4/\sqrt{50}]^2$ in K\"ahler normal coordinates around $\xi=[1:0],$ the unique point for which $d\phi_{z_0}(\xi)=0.$ Near each critical point at $w$ with $\abs{w}>1,$ the sector predicted in Theorem \ref{T:Sendov} (with parameter $c=3/4$) to contain its paired zero is shown.}
\end{figure}

\begin{figure}[h]
\centering
\includegraphics[scale=.75]{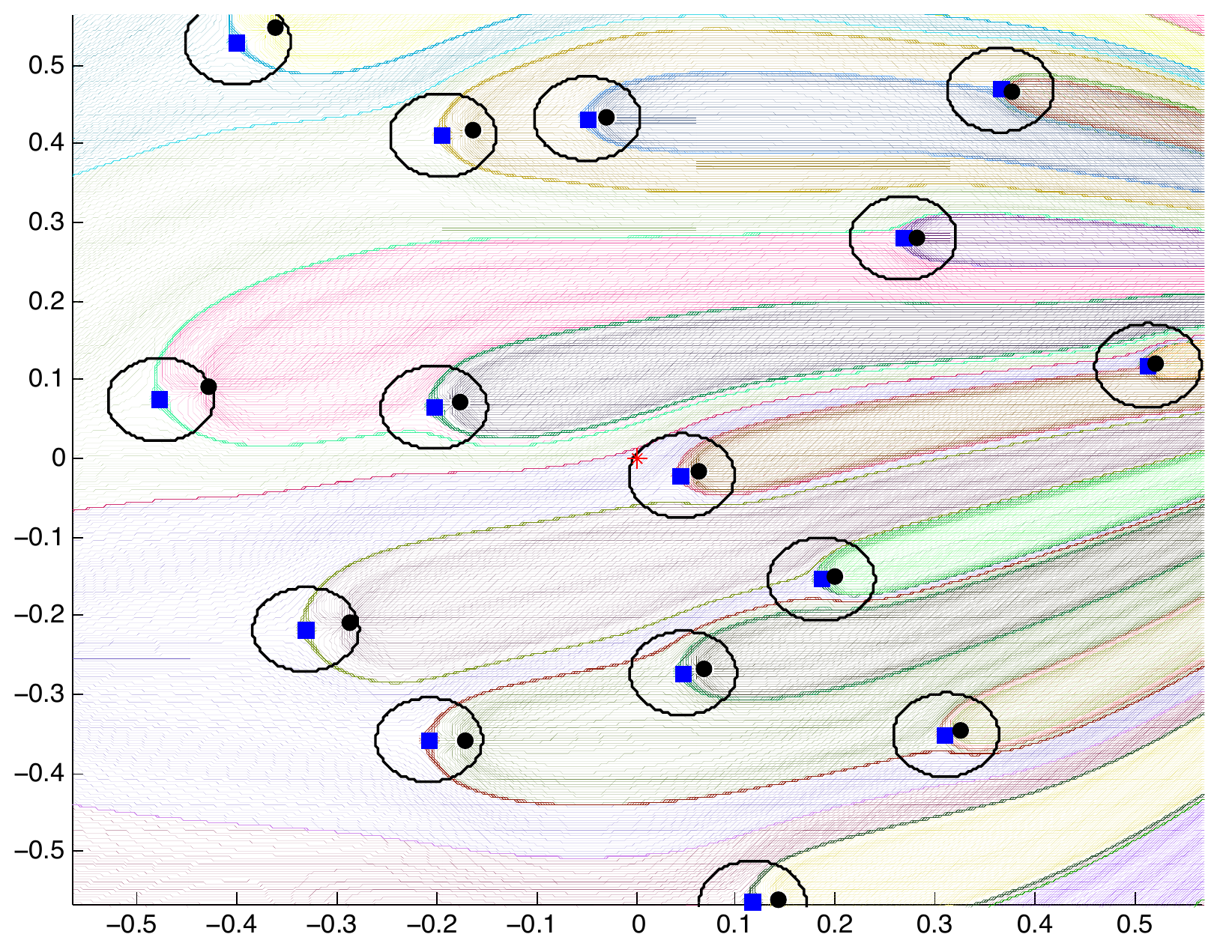}
\label{F:Scaled SU2 Balls NonCrit}
\caption{Zeros (black discs) and holomorphic critical points (blue squares) for an $SU(2)$ polynomial $p_{50}$ of degree $50$ inside [$-4/\sqrt{50},4/\sqrt{50}]^2$ in K\"ahler normal coordinates around a generic point $\xi=2+.3 i,$ which is shown as a red asterisk. Around each critical point, the circle of radius $N^{-3/4},$ predicted by Theorem \ref{T:Sendov} to contain its paired zero, is shown.}
\end{figure}

\subsection{Informal Discussion of Results}\label{S:Informal} 
Our main result is Theorem \ref{T:Crit Two Point Function}. Together with Theorem \ref{T:EZ}, it gives an asymptotic formula for $K_N$ in local coordinates near any $\xi$ on the Riemann sphere. More precisely, we compactify $\C$ into the Riemann sphere $\C P^1\cong S^2$ and fix $\xi\in \C P^1.$ To resolve individual zeros and critical points appearing near $\xi$ and study their correlations, we work in a particular holomorphic coordinate, called a K\"ahler normal coordinate, centered at $\xi$ and dilate by a factor of $N^{1/2}$ relative to $\xi$ (cf Definition \ref{D:Scaled KNC}). Our choice of coordinate is adapted to $h$ and gives a universal yardstick for measuring local correlations (cf Section \ref{S:Scaling}). The $N^{1/2}$ scaling compensates for the typical  $N^{-1/2}$ distance between $N$ well-spaced points on $\C P^1.$ 

Theorem \ref{T:Crit Two Point Function} shows that correlations between zeros and critical points in scaled coordinates near $\xi$ depend strongly on whether $d\phi_{z_0}(\xi)=0$ ($\phi_{z_0}$ is defined in (\ref{E:Special KPot})). Namely, when $d\phi_{z_0}(\xi)=0,$ zeros and critical points are highly correlated but stay a bounded distance apart and tend to appear in rigid pairs (cf Figure 2). When $d\phi_{z_0}(\xi)\neq 0,$ however, zero and critical point pairs are separated by a distance of at most $N^{-1/4}$ and hence coincide in the large $N$ limit (see Figure 3). This happens because the leading term in powers of $N$ for the connection ${\nabla^{z_0}}^{\otimes N}$ in scaled coordinates near such $\xi$ is an order $0$ differential operator. Zeros and critical points thus become indistinguishable. This can be seen directly from (\ref{E:Two Connections}) and is explained in Section \ref{S:Connection}.

In Theorem \ref{T:Sendov}, our main application of Theorem \ref{T:Crit Two Point Function}, we study nearest neighbor spacings between zeros and critical points in scaled local coordinates near a fixed $\xi \in \C P^1.$ We fix a measurable set $A$ in these coordinates and show that the expected number of critical points lying in $A$ is equal to the expected number of zero and critical point pairs $(z,w)$ with $w\in A$ and $z$ ``paired'' with $w$ is a nearly deterministic way. It is tempting to interpret this result by saying that, on average, each critical point comes paired with a unique zero. Although this interpreation is plausible from Figures 1-3, Theorem \ref{T:Sendov} is consistent with the possibilities such as, on average, half the critical points of $p_N$ being paired with two zeros and half are not being paired with any zeros. Developing the tools to exclude such possibilities is work in progress by the author. 

In addition to the cross correlation current $K_N,$ we fix $\xi \in \C P^1$ and treat in Theorem \ref{T:Crit Given Zero One Point Function} the conditional current
\[\E{C_{p_N}\,|\, p_N(\xi)=0}\]
in the sense of Shiffman, Zelditch, and Zhong in \cite{Conditional}. As explained in Section 6.1 of \cite{Conditional}, $\E{C_{p_N}\,|\, p_N(\xi)=0}$ gives a measure of the correlations between zeros and critical points that is quite different from the conditional density obtained from $K_N.$ Finally, in Theorem \ref{T:EZ}, we give local and global asymptotics for the (unconditional) expected distribution of critical points $\E{C_{p_N}}.$ 

\subsection{Formal Statement of Results}\label{S:Formal}
To state our results, we introduce the covariance current
\[\Cov_N:=\E{C_{p_N}\wedge Z_{p_N}}-\E{C_{p_N}}\wedge \E{Z_{p_N}}.\]
The currents $\E{C_{p_N}}$ and $\E{Z_{p_N}}$ are relatively simple (see Theorem \ref{T:EZ} and Remark \ref{R:Zeros One Point Function}) so that the study of $K_N$ and $\Cov_N$ are essentially equivalent. We also define 
\[G(t):=\frac{\gamma}{4}-\frac{1}{4}\int_0^{t^2}\frac{\log(1-s)}{s}ds\qquad t\in [0,1),\]
where $\gamma$ is the Euler-Macheroni constant. Finally, let $\xi \in \C P^1$ and fix $N\geq 1.$ Consider any neighborhood $U$ of $z$ and any holomorphic coordinate $z:U\gives \C$ such that $z(\xi)=0$ and $\w_h(z)=dz\wedge d\cl{z}+o(\abs{z}^2).$ ($z$ is called a K\"ahler normal coordinate and exists on any K\"ahler manifold). We make the following 
\begin{definition}\label{D:Scaled KNC}
We define $u:U\gives \C$ given by 
\[u=u(z):=z\cdot N^{1/2}\] 
to be a $N^{-1/2}-$scale normal coordinate at $\xi.$
\end{definition}
\begin{Thm}[Covariance Current Asymptotics] \label{T:Crit Two Point Function}
Let $h$ be a smooth positive Hermitian metric on $\mathcal O(1),$ and suppose $p_N$ is a degree $N$ polynomial drawn from the Hermitian Gaussian ensemble corresponding to $h.$ \\
\noindent {\bf 1. Global Asymptotics. } For each $N\geq 1,\, \ep>0$ and every $\psi \in C^4(\C P^1\x \C P^1)$
  \begin{equation}
    \label{E:Global PCF}
\left(\frac{1}{N(N-1)}\Cov_N(z,w),\, \psi(z,w)\right)=O(N^{-2+\ep}).
  \end{equation}
The implied constant depends only on $\ep$ and $\psi.$ \\
\noindent {\bf 2. Local Asymptotics. } Fix $N\geq 1,$ $\xi \in \C P^1,$ and a $N^{-1/2}-$scale normal coordinate centered at $\xi.$ For every $\ep>0$
\begin{equation}
  \label{E:Cov BiPot}
\Cov_N(z,w)=\left(\frac{i}{2\pi}\partial_z \cl{\partial}_z\wedge \frac{i}{2\pi}\partial_w \cl{\partial}_w\right)G(P_{\xi}(z,w))+O(N^{-1/2+\ep}),  
\end{equation}
where
\begin{equation}\label{E:PN Scaling Limit}
  P_{\xi}(z,w):=
\begin{cases}
  \case{e^{-\frac{1}{2}\abs{z-w}^2}}{d\phi_{z_0}(\xi)\neq 0\text{ or }\xi=\infty}\\
\case{\frac{\abs{\cl{w}\cdot \DDi{\cl{w}}{\phi_{z_0}}\big|_{\xi}+z}}{\sqrt{1+\abs{\cl{w}\cdot \DDi{\cl{w}}{\phi_{z_0}}\big|_{\xi}+w}^2}}e^{-\frac{1}{2}\abs{z-w}^2}}{d\phi_{z_0}(\xi)=0}
\end{cases}.
\end{equation}
The implied constant in (\ref{E:Cov BiPot}) depends on $\ep,$ and the expression may be paired with any bounded measurable function.
\end{Thm}
\begin{remark}\label{R:Same as Zeros}
If $d\phi_{z_0}(\xi)\neq 0$ or $\xi=\infty,$ then our local formula for $\Cov_N$ is identical to the formula for variance current $\E{Z_{p_N}\wedge Z_{p_N}}-\E{Z_{p_N}}\wedge \E{Z_{p_N}}$ derived in \cite{NV}. This may seem surprising. It is a consequence, as explained in Section \ref{S:Connection}, of the fact that the large $N$ limit of the derivative $\frac{d}{dz}$ in scaled coordinates near a point $\xi$ satisfying $d\phi_{z_0}(\xi)\neq 0$ or $\xi=\infty$ is an order $0$ differential operator. Zeros and critical points therefore become indistinguishable. In contrast, when $d\phi_{z_0}(\xi)=0,$ the scaling limit of $\frac{d}{dz}$ is a genuine order $1$ differential operator. 
\end{remark}

\noindent Theorem \ref{T:Sendov}, which we state next, is our main application of Theorem \ref{T:Crit Two Point Function}. Cases 1 and 2 are illustrated in Figures 3 and 2, respectively.  

\begin{Thm}[Expected Nearest Neighbor Spacings]\label{T:Sendov}
With the notation of Theorem \ref{T:Crit Two Point Function}, fix $\xi \in \C P^1$ and consider a $N^{-1/2}-$scale normal coordinate centered at $\xi.$ Fix a bounded $A\subset \C$ and write
\[C_A:=\#\setst{w\in A}{\frac{d}{dw}p_N(w)=0}\qquad \text{and}\qquad Z_A:=\#\setst{z\in A}{p_N(z)=0}.\]
\noindent {\bf Case 1.} Suppose that $d\phi_{z_0}(\xi)\neq 0$ or $\xi=\infty.$ Define the random variables $\curly X_{A,N}$ to be the number of pairs $(z,w)\in \C\x \C$ such that 
  \[w\in A, \qquad p_N(z)\,=\,\frac{d}{dw}p_N(w)\,=\,0,\qquad\text{and}\qquad \abs{z-w}\leq N^{-1/4}.\]
Then, for each $\ep>0,$
\begin{equation}
  \label{E:Sendov NonCrit}
\E{\curly X_{A,N}}=\E{C_A}+O(N^{-1/4+\ep})=\E{Z_A}+O(N^{-1/4+\ep}).  
\end{equation}
\noindent {\bf Case 2.} Suppose now that $d\phi_{z_0}(\xi)=0.$ Define $\alpha:= \DDi{\cl{w}}{\phi_{z_0}}(\xi)$ and introduce $\zeta:=\alpha \cl{w}+w.$ Assume that $A\subseteq \C\backslash \set{\abs{\zeta}\leq 1},$ and fix a parameter $c\in (2/3,1).$ Let $\curly X_{A,N,c}$ be the number of pairs $(z,w)\in \C\x \C$ such that 
\[w\in A, \qquad p_N(z)\,=\,\frac{d}{dw}p_N(w)\,=\,0,\qquad\text{and}\qquad z=w+re^{it}\]
with 
\begin{equation}
  \label{E:Constraints}
\qquad \qquad r\in \left[\abs{\zeta}^{-1}-\abs{\zeta}^{-1-c}, \abs{\zeta}^{-1}+\abs{\zeta}^{-1-c}\right]\quad \text{and}\quad t\in \left[\arg(\zeta)-\abs{\zeta}^{-c}, \arg(\zeta)+\abs{\zeta}^{-c}\right].  
\end{equation}
We have for each $\ep>0$
\begin{equation}
  \label{E:Crit Sendov Pairing}
\E{\curly X_{A,N,c}}=\E{C_A}+O\left(\int_A \abs{\zeta}^{2-3c}dw \wedge d\cl{w}\right) +O(N^{-1/2+\ep}),
\end{equation}
where the implied constant in the first error term depends only on $c$ and in the second error term depends only on $\ep.$
\end{Thm}
\begin{remark}
The situation in Case 1 of Theorem \ref{T:Sendov} is the generic behavior. Indeed, the positivity of $h$ means that $\phi_{z_0}$ is subharmonic and so can vanish at only finitely many points in any bounded subset of $\C P^1\backslash\set{\infty}.$ Moreover, $\phi_{z_0}$ must take the form $\log(1+\abs{z}^2)+\psi(z)$ for some $\psi$ that is smooth function on all of $\C P^1.$ Since the derivative of $\psi$ is bounded while the derivative of $\log(1+\abs{z}^2)$ is unbounded at infinity, $d\phi_{z_0}\neq 0$ in some neighborhood of $\infty.$ 
\end{remark}
\begin{remark}
In the case when $d\phi_{z_0}(\xi)=0,$ the constrains (\ref{E:Constraints}) form a tight sector in polar coordinates around each critical point. The radial and angular widths of this sector decrease as $\abs{w}$ grows and are determined by the $2-$jet of $\phi_{z_0}$ at $\xi$ via the paramter $\alpha.$ Different values of $\alpha$ can cause rather different kinds of pairings. For example, $\alpha=0$ means $\zeta=w,$ while $\alpha=1$ means $\zeta=\Re w.$
\end{remark}

\noindent In addition to studying the covariance current $\Cov_N,$ we study for any fixed $\xi\in \C P^1$ the conditional $(1,1)-$current
\[\E{C_{p_N}\,|\, p_N(\xi)=0}\]
in the sense of Shiffman, Zelditch, and Zhong in \cite{Conditional}. Since the event $p_N(\xi)=0$ has probability $0,$ we specify that the particular random variable used to define the conditional expectation is the evaluation map $ev_{\xi}$ at $\xi.$ Continuing the trend of Theorems \ref{T:Crit Two Point Function} and \ref{T:Sendov}, we see in Theorem \ref{T:Crit Given Zero One Point Function} that near $\xi$ satisfying $d\phi_{z_0}(\xi)\neq 0$ or $\xi=\infty,$ the current $\E{C_{p_N}\,|\, p_N(\xi)=0}$ is indistinguishable from the current $\E{Z_{p_N}\,|\, p_N(\xi)=0}$ studied in \cite{Conditional}. Near $\xi$ satisfying $d\phi_{z_0}(\xi)=0,$ however, a different behavior emerges. 

\begin{Thm} \label{T:Crit Given Zero One Point Function}
Let $h$ be a smooth positive Hermitian metric on $\mathcal O(1)\twoheadrightarrow \C P^1$ and suppose $p_N$ is a degree $N$ polynomial drawn from the Hermitian Gaussian ensemble corresponding to $h.$ Fix $\xi \in\C P^1.$\\
\noindent {\bf 1. Global Asymptotics.}  For each $N\geq 1$ and every $\psi \in C^{\infty}(\C P^1),$
  \begin{equation}
    \label{E:Global OPF}
\left(\E{C_{p_N}\,|\,p_N(\xi)=0},\psi\right)=N\cdot \left(\E{C_{p_N}},\psi\right)+O(1).
    \end{equation}
The implied constant is independent of $N.$ \\
\noindent {\bf 2. Local Asymptotics. } Fix $\zeta \in \C P^1$ and a $N^{-1/2}-$scale normal coordinates centered at $\zeta.$ If $\xi\neq \zeta,$ then for each $k\geq 1.$
\begin{equation}
  \label{E:Conditional at Diff Pt}
  \E{C_{p_N}\,|\, p_N(\xi)=0}(w)=\E{C_{p_N}}(w)+O(N^{-k}).
\end{equation}
Suppose now $\xi=\zeta.$ Then, if $d\phi_{z_0}(\xi)\neq 0$ or $\xi=\infty,$ we have for any $\ep>0$
\begin{equation}
  \label{E:Crit Given Zero Scaled I}
  \E{C_{p_N}\,|\, p_N(\xi)=0}(u) = \E{C_{p_N}}(u)+\frac{i}{2\pi}\partial_u\cl{\partial}_u\log(1-e^{-\abs{u}^2}) + O(N^{-1/2+\ep}).
\end{equation}
Finally, if $\xi=\zeta$ and $d\phi_{z_0}(\xi)=0,$ then $\E{C_{p_N}\,|\, p_N(\xi)=0}(u)$ is the smooth $(1,1)-$form
\begin{equation}\label{E:Crit Given Zero Scaled II}
\frac{i}{2\pi}\partial_u\cl{\partial}_u\left( \log (1-P_{\xi}(u,0)^2)+\log\left(1+\abs{\DDi{u}{\phi_{z_0}}\bigg|_0\cdot  u+\cl{u}}^2\right)\right)+O(N^{-1/2+\ep}),
\end{equation}
where, as in Theorem \ref{T:Crit Two Point Function}, 
\[P_{\xi}(u,0)=\frac{\abs{\cl{u}\cdot \DDi{\cl{u}}{\phi_{z_0}}\big|_{\xi}}}{\sqrt{1+\abs{\cl{u}\cdot \DDi{\cl{u}}{\phi_{z_0}}\big|_{\xi}+u}^2}}e^{-\frac{1}{2}\abs{u}^2}.\]
\end{Thm}
\begin{remark}
 Expression (\ref{E:Crit Given Zero Scaled I}) may be written explicitly as
\begin{equation}
  \label{E:Crit Given Zero Scaled I Explicit}
  \E{C_{p_N}\,|\, p_N(\xi)=0}(u) = \pi \delta_0(u) + \left(\frac{1-(1+\abs{u}^2)e^{-\abs{u}^2}}{(1-e^{-\abs{u}^2})^2}+1\right)\frac{i}{2\pi}\partial_u\cl{\partial}_u\abs{u}^2 + O(N^{-1/2+\ep}),
\end{equation}
and it coincides precisely with the expression obtained in \cite{Conditional} for $N^{-1/2}-$ scale current of $\E{Z_{p_N}\,|\,p_N(\xi)=0}.$ The reason is just as in Remark \ref{R:Same as Zeros}.
\end{remark}

\noindent Finally, we state our result about the local and global asymptotics for $\E{C_{p_N}}(w).$
\begin{Thm} \label{T:EZ}
Let $h$ be a smooth positive Hermitian metric on $\mathcal O(1)\twoheadrightarrow \C P^1$ and suppose $p_N$ is a degree $N$ polynomial drawn from the Hermitian Gaussian ensemble corresponding to $h.$ \\
\noindent {\bf 1. Global Asymptotics} For each $N\geq 1$ and every $\psi \in C^{\infty}(\C P^1)$
  \begin{equation}
    \label{E:Crit Global One Point}
\left(\E{C_{p_N}},\psi\right)=N\cdot (\w_h,\psi)+O(1)
  \end{equation}
with the implied constant depending only on $\psi.$ \\
\noindent {\bf 2. Local Asymptotics. } Fix $\xi \in \C P^1$ and scaled normal coordinate centered at $\xi.$ For each $N\geq 1$ we may write 
\begin{equation}
  \label{E:Local 1PF NonCrit}
\E{C_{p_N}}(w)= \frac{1}{\pi}\partial \cl{\partial}\abs{w}^2+O(N^{-1/2}).  
\end{equation}
if $d\phi_{z_0}(\xi)\neq 0$ or $\xi=\infty.$ If, on the other hand, $d\phi_{z_0}(\xi)=0,$ then  
  \begin{equation}
    \label{E:Local 1PF Crit}
\E{C_{p_N}}(w)= \frac{i}{2\pi}\partial \cl{\partial}\log \left(1+\abs{\DDi{w}{\phi_{z_0}}\bigg|_{\xi}w+\cl{w}}^2\right)+\frac{1}{\pi}\partial \cl{\partial}\abs{w}^2+O(N^{-1/2}).
  \end{equation}
Both expressions (\ref{E:Local 1PF Crit}) and (\ref{E:Local 1PF NonCrit}) may be paired with any bounded measurable function.  
\end{Thm}
\begin{remark}\label{R:Zeros One Point Function}
It is an easy consequence of Lemma \ref{L:PLL} below (and was proved as Lemma 3.1 in \cite{Quantum} for example) that in $N^{-1/2}-$scale normal coordinate centered at any $\xi\in \C P^1,$ we have
\[\E{Z_{p_N}}(z)= \frac{1}{\pi}\partial \cl{\partial}\abs{z}^2+O(N^{-1/2}),\]
which coincides with (\ref{E:Local 1PF NonCrit}). The reason that zeros and critical points have the same expected distribution near $\xi$ satisfying $d\phi_{z_0}(\xi)=0$ or $\xi=\infty$ is the same as in Remark \ref{R:Same as Zeros}. 
\end{remark}

\subsection{Smooth Versus Holomorphic Critical Points}\label{S:Smooth VS Holomorphic}
In order to put the current work in perspective, we contrast our purely holomorphic notion of critical points with the smooth critical points studied in \cite{VacuaI, VacuaII, VacuaIII, Transportation}. Let $p_N$ be a random polynomial of degree $N$ drawn from the Hermitian Gaussian ensemble corresponding to a fixed smooth positive hermitian metric $h$ on $\mathcal O(1).$ 

As explained in Section \ref{S:CP1}, it is natural to view $p_N(z)$ as a holomorphic section of the line bundle $\mathcal O(N)\twoheadrightarrow \C P^1.$ Critical points for sections of lines bundles depend on the choice of a connection. A natural choice of connection on $\mathcal O(N)$ is the $N^{th}$ tensor power of $\nabla^h,$ the metric connection compatible with $h.$ Critical points with respect to $\nabla^h$ are what we refer to as smooth critical points. Holomorphic critical points correspond to the meromorphic connection $\nabla^{z_0}$ on $\mathcal O(1)$ defined in (\ref{E:Two Connections}) (cf Section \ref{S:Connection}). Thus, while holomorphic critical points are solutions of $\frac{d}{dz}p_N(z)=0,$ smooth critical points are solutions of $\frac{d}{dz}\left[p_N(z)\cdot \norm{z_0(z)}_h^{-N}\right]=0.$ 

With respect to the usual frame $z_0$ of $\mathcal O(1)$ over $\C P^1\backslash\set{\infty}$ we deduce from equation (\ref{E:Two Connections})
\[{\nabla^{z_0}}^{\otimes N}=\nabla^h-N\cdot \partial \phi_{z_0}.\]
The condition $d\phi_{z_0}(\xi)\neq 0$ implies that $\nabla^{z_0}$ becomes an order $0$ operator to top order in $N$ locally around $\xi.$ Zeros and critical points therefore become indistinguishable in the large $N$ limit. This accounts for the importance of the condition $d\phi_{z_0}(\xi)=0$ in Theorems \ref{T:Crit Two Point Function}-\ref{T:EZ}. 

\subsection{Acknowledgements} 
This work was suggested by Steve Zelditch, whose patient explanations about his work on Sz\"ego kernels and scaling limits for correlations of zeros have been invaluable. I would also like to thank Dean Baskin, Leonid Hanin, Eric Potash, Pokey Rule, Josh Shadlen, and Jared Wunsch for many useful discussions. Finally, I am indebted to Manjunath Krishnapur and Ron Peled for kindly sharing with me the Matlab code which I modified to produce the above figures.

\subsection{Outline of Paper}
The rest of the paper is organized as follows. First, in Sections \ref{S:SU2} and \ref{S:BF}, we introduce two model ensembles of Gaussian analytic functions that appear naturally in our work: the $SU(2)$ polynomials and the Bargmann-Fock random analytic functions. In Section \ref{S:CP1}, we recall some basic complex geometry and establish notation that will be used throughout. Next, in Section \ref{S:Scaling}, we recall some notions from K\"ahler geometry and introduce the scaling limits of Theorems \ref{T:Crit Two Point Function}-\ref{T:EZ}. We then recall in Section \ref{S:BS} how Sz\"ego kernels and their off-diagonal asymptotics are analyzed. We derive the scaling asymptotics for $\frac{d}{dz}$ in Section \ref{S:Connection} and obtain as a consequence asymptotics for holomorphic derivatives of the Sz\"ego kernels in Section \ref{S:BS Asymptotics Derivs}. The computations in this section are the extra ingedients needed to apply the methods of \cite{PLL} and \cite{NV} to our problem. In Section \ref{S:Relation}, we recall and provide a proof of an important lemma that relates Sz\"ego kernels to the distribution of zeros and critical points for holomorphic sections of line bundles in general. Finally, in Sections \ref{S:Proof 1pt}-\ref{S:Proof Sendov}, we give proofs of our results. 

\section{Background}\label{S:Background}
\noindent We begin by introducing in Sections \ref{S:SU2} and \ref{S:BF} two important and well-studied ensembles of random analytic functions: the $SU(2)$ polynomials and the Bargmann-Fock random analyltic functions. 

\subsection{$SU(2)$ Polynomials}\label{S:SU2}
The $SU(2)$ polynomials are the most computationally tractable of the Hermitian Gaussian ensembles. We introduce them for several reasons. First, we wish to derive the results of Theorem \ref{T:EZ} without appealing to any difficult Sz\"ego Kernel asymptotics, which are necessary to treat the case of a general Hermitian Gaussian ensembles. Second, we wish to illustrate how the distribution of zeros and critical points near some $\xi \in \C P^1$ depends on whether $d\phi_{z_0}(\xi)$ vanishes. We mention that Feng and Wang in \cite{Renjie} study the distribution of critical values (rather than critical points) for $SU(2)$ polynomials. 

In the language of our paper, $SU(2)$ polynomials are the Hermitian guassian ensemble corresponding to the Fubini-Study metric $h_{FS}$ on $\mathcal O(1).$ Over a point $[z_0:z_1]\in \C P^1,$ the meric $h_{FS}([z_0:z_1])$ is obtained by restricting the usual Hermitian inner product on $\C^2$ to the complex line $\mathcal O(-1)|_{[z_0:z_1]}$ and considering the dual metric. The name $SU(2)$ polynomials comes from the fact that $h_{FS}$ is invariant under the natural action of $SU(2)$ on $\C P^1.$ We denote as in Section \ref{S:CP1} by 
\[[z_0:z_1]\mapsto\frac{z_1}{z_0}=:z\]
the standard affine coordinate on $\C P^1-\set{\infty}$ and by $z_0^N=z_0^{\otimes N}$ the standard frame for $\mathcal O(N)$ over $\C P^1\backslash\set{\infty}.$ A degree $N$ $SU(2)$ polynomial (section) is given explicitly by the formula
\begin{equation}\label{E:SU2 Def}
p_N(z):=\left[\sum_{j=0}^N a_j\left((N+1)\binom{N}{j}\right)^{1/2}z^j\right]\cdot z_0^N,~~a_j\sim N(0,1)_{\C} ~~ i.i.d.
\end{equation}
The individual sections 
\[S_j(z):=\left((N+1)\binom{N}{j}\right)^{1/2}z^j\cdot z_0^N\]
are orthonormal for the inner product (\ref{E:Inner Product}). Indeed, in the coordinate $z,$ this inner product may be written as follows:
\begin{equation}
  \label{E:SU2 IP}
\inprod{s_1}{s_2}=\int_{\C} \frac{f_1(z)\cl{f_2(z)}}{(1+\abs{z}^2)^{2N+2}}\cdot \frac{i}{2\pi}dz\wedge d\cl{z}, \quad s_j=f_j\cdot z_0^N\in H_{hol}^0(\C P^1,\mathcal O(N)),\,\, j=1,2.  
\end{equation}
We have used that $h_{FS}(z)=(1+\abs{z}^2)^{-1}$ so that 
\[\w_h=\frac{i}{2\pi}\partial_z\cl{\partial}_z\log h^{-2}=\frac{idz\wedge d\cl{z}}{2\pi(1+\abs{z}^2)^2}\]
and
\[\norm{z_0}_{h_{FS}}^{2N}=(1+\abs{z}^2)^{-2N}.\]
That the monomials $\set{z^j}_{j=0}^N$ are orthogonal follows immediately by passing to polar coordinates in (\ref{E:SU2 IP}) and is true by the same argument for any toric metric that is equivariant with respect to the particular $S^1$ action on $\C P^1$ that fixes $[1:0]$ and $[0:1],$ the north and south poles. The $L^2$ norms 
\[\norm{z^j}_{L^2}^2=\frac{1}{(N+1)\binom{N}{j}}\]
may be computed by lifting $\abs{z}^{2j}$ to $S^1$ equivariant functions on the principle $S^1$ bundle $\pi:S^3\twoheadrightarrow \C P^1$ (the Hopf Fibration) and using that the Fubini-Study metric $\w_{FS}$ on $\C P^1$ is the pushforward under $\pi$ of the round metric on $S^3.$ Integration against the round metric on $S^3$ can then be further lifted to a guassian integral on $\C^2.$ See Section 1.3 in \cite{Universality} for more detials. 

\subsubsection{Expected Global Distribution of Zeros and Critical Point} Since $h_{FS}$ is invariant under the full $SU(2)$ group of isometries of $(\C P^1, \w_{FS})$, we immediately deduce that
\[\E{Z_{p_N}}=N\cdot \w_{FS}.\]
We may see this alternatively from Lemma \ref{L:Zero Expected Density}, which says that
\[\E{Z_{p_N}}(z)=\frac{i}{2\pi}\partial_z\cl{\partial}_z \log \left[(N+1) \sum_{j=0}^N \binom{N}{j}\abs{z}^{2j}\right]=\frac{i}{2\pi}\partial_z\cl{\partial}_z \log \left[(1+\abs{z}^2)^N\right].\]
Performing the differentiation yields 
\begin{equation}\label{E:SU2 Zeros}
\E{Z_{p_N}}(z)=\frac{N}{\pi}\left[\frac{1}{(1+\abs{z}^2)^2}\right]\frac{i}{2}dz\wedge \cl{dz},
\end{equation}
as expected. Similarly, by Corollary \ref{C:Crit Expected Density}, we see that
\[\E{C_{p_N}}(w)=\frac{i}{2\pi}\partial_w\cl{\partial}_w\log\left[\frac{\partial^2}{\partial w\partial \cl{w}}\left( 1+\abs{w}^2\right)^N\right].\]
Again, we may perform the differentiation explicity to obtain
\begin{equation}\label{E:SU2 Crits}
\E{C_{p_N}}(w)=\frac{N}{\pi}\left[ \left(\left(1-\frac{2}{N}\right)\frac{1}{(1+\abs{w}^2)^2}+\frac{1}{(1+N\abs{w}^2)^2}\right)\right]\frac{i}{2}dw\wedge \cl{dw}.
\end{equation}
Equation (\ref{E:SU2 Crits}) recovers the results of Macdonald in \cite{Mac}. The expected distribution of the zeros of $p_N$ is $N$ times Fubini-Study measure on $\C P^1$ for \textit{every} N. In contrast, the critical points are only distributed uniformly on $\C P^1$ in the large $N$ limit. This is not surprising in light of the Guass-Lucas Theorem, which asserts that the holomorphic critical points of any complex polynomial lie inside the convex hull of its zeros. 

\subsubsection{Local Distribution of Zeros and Critical Points} We study the local behavior of zeros and critical points near some $\xi\in \C P^1$ in the $N^{-1/2}-$scale normal coordinates around $\xi$ of Definition 1 (cf also Section \ref{S:Scaling}). The curvature of $h_{FS}$ is $\w_{FS},$ the Fubini-Study metric on $\C P^1.$ K\"ahler normal coordinates at $\xi=[\alpha_0:\alpha_1]\in \C P^1$ for $\w_{FS}$ are given by the usual affine coordinate on $\mathcal O(1)|_{\C P^1-\set{-\cl{\alpha}_1:\cl{\alpha_0}}}.$ For example, in the standard (affine) coordinate centered at $[1:0]\in \C P^1$ and with respect to the standard frame $z_0$ of $\mathcal O(1),$ we have that 
\[\w_{FS}(z)=\partial \cl{\partial}\log \norm{z_0}_{h_{FS}}^{-2}(z)=\partial \cl{\partial}\log (1+\abs{z}^2)=dz\wedge d\cl{z}+o(\abs{z}^2),\]
as required. Fix some $\xi =[\xi_0:\xi_1]\neq [0:1].$ K\"ahler normal coordinates around $\xi$ are given by
\[\zeta:[z_0:z_1]\mapsto \frac{\xi_0z_1-\xi_1 z_0}{\cl{\xi_1}z_1+\cl{\xi_0}z_0},\]
the ratio of the sections of $\mathcal O(1)$ that vanish to order $1$ at $[\xi_0:\xi_1]$ and at the antipodal point $[-\cl{\xi_1}:\cl{\xi_0}].$ The $N^{-1/2}-$scale K\"ahler normal coordinates of limits of Theorems \ref{T:Crit Two Point Function}-\ref{T:EZ} around $\xi=[\xi_0:\xi_1]$ are then obtained by rescaling
$$\tau_N^{\xi}(\zeta):=\zeta\cdot N^{1/2}.$$
We see by direct computation from (\ref{E:SU2 Zeros}) that
$$\E{(\tau_N^{\xi})_*Z_{p_N}}(u)\gives \partial_u \cl{\partial}_u \frac{i}{2\pi}\abs{u}^2,$$
in accordance with Remark \ref{R:Zeros One Point Function}. Similarly, from (\ref{E:SU2 Crits}), we find that 
$$\E{(\tau_N^{\xi})_*C_{p_N}}(u)\gives \frac{i}{2\pi}\partial_u \cl{\partial}_u\left(\log\left(1+\abs{u}^2\right)+\abs{u}^2\right)$$
if $\xi=[1:0]$ and 
\[\E{(\tau_N^{\xi})_*C_{p_N}}(u)\gives \partial_u \cl{\partial}_u \frac{i}{2\pi}\abs{u}^2\]
otherwise. Since $d\phi_{z_0}(\xi)=0$ for the Fubini-Study metric if and only if $\xi=0,$ this recovers the local asymptotics of $\E{C_{p_N}}$ from Theorem \ref{T:EZ} for special case of $h=h_{FS}.$

\subsection{Bargmann-Fock}\label{S:BF} The Bargmann-Fock Space $\curly F:=L^2(\C,\frac{1}{\pi}e^{-\abs{z}^2}dz)\intersection \mathcal O_{hol}(\C)$ consists of the entire functions that are square integrable with respect to the standard gaussian measure $\frac{1}{\pi}e^{-\abs{z}^2}dz$ on $\C.$ An orthonormal basis for $\curly F$ with respect to the induced $L^2$ inner product is $\set{\frac{z^j}{\sqrt{j!}}}_{j\geq 0}$ and a Bargmann-Fock random analytic function, sometimes referred to as a Gaussian Entire Function or a Gaussian Analytic Function, is
\[f(z):=\sum_{j=0}^N a_j \frac{z^j}{\sqrt{j!}},\quad a_j\sim N(0,1)_{\C}\,\, i.i.d.\] 
The Bargmann-Fock random analytic functions have been extensively studied in \cite{KrishThesis, UniverCorr,Transportation}. In the context of our work, we think of $\curly F$ as the space of $L^2$ holomorphic sections of the trivial line bundle $\C\x \C$ endowed with the Hermitian metic
\[\norm{(z,1)}_{h_{BF}}^2:=\frac{1}{\pi}e^{-\abs{z}^2}.\]
Here $(z,1)$ is the constant trivializing section of $\C\x\C.$ The $L^2$ inner product on $\curly F$ is then 
\begin{equation}
  \label{E:BF KP}
\inprod{s_1}{s_2}=\int_{\C} h_{BF}(s_1(z),s_2(z))w_{h_{BF}}(z)  
\end{equation}
in complete analogy with (\ref{E:Inner Product}). 

It is an important observation that the Bargmann-Fock ensemble is the local scaling limit for all the Hermitian Gaussian ensembles. To see this, we choose $\xi\in \C P^1$ and a $N^{-1/2}-$scale normal coordinate centered at $\xi.$ As explained in Section \ref{S:Scaling}, the K\"ahler potential for the metric $h^N$ then takes the form 
\[\abs{z}^2+O(N^{-1/2}),\]
which coincides to leading order with the potential for $h_{BF}.$ Moreover, thought of as a holomorphic Gaussian field, $f$ is characterized by its covariance kernel
\[Cov_{BF}(z,w)=\E{\norm{f(z)\otimes \cl{f(w)}}_{h_{BF}}}=e^{z\cdot \cl{w}-\frac{1}{2}(\abs{z}^2+\abs{w}^2)}.\]
Part of the content of the $C^{\infty}$ asymptotic expansion for the Sz\"ego kernels (see Section \ref{S:BS}) is that in $N^{-1/2}-$scale K\"ahler normal coordinates around any $\xi \in \C P^1$, the covariance kernels 
\[\Pi_N(z,w)=\E{\norm{p_N(z)\otimes \cl{p_N}(w)}_{h^N}}\]
for Hermitian Gaussian random sections of $\mathcal O(N)$ converge in the $C^k$ topology to $Cov_{BF}(z,w)$ for all $k.$ One may see this concretely for $SU(2)$ polynomials by applying Sterling's formula to $\binom{N}{j}$ in equation (\ref{E:SU2 Def}) and rescaling $z\mapsto \frac{z}{\sqrt{N}}$ to obtain 
\[\sum_{j=0}^N a_j\sqrt{(N+1)\binom{N}{j}}\left(\frac{z}{\sqrt{N}}\right)^j\approx \sum_{j=0}^N a_j\frac{z^j}{\sqrt{j!}},\]
which is the truncated Bargmann-Fock random analytic function. 

\subsection{Complex Projective Space}\label{S:CP1}
We recall some basic facts about $\C P^1$ and introduce some notation. By definition, $\C P^1$ is the space of complex lines through the origin in $\C^2.$ Each line is determined by a pair $(z_0,z_1)\in \C ^1\backslash \set{(0,0)}.$ We denote by $[z_0:z_1]$ the equivalence class of pairs $(z_0,z_1)$ that determine the same line. The notation $[z_0:z_1]$ is called homogenous coordinates. We will refer to $[1:0]$ and $[0:1]$ variously as the south and north poles or at $0$ and $\infty,$ respectively.
 
The tautological line bundle $\mathcal O(-1)\twoheadrightarrow \C P^1$ assigns to each $[z_0:z_1]$ the line in $\C ^2$ passing through $(z_0,z_1)$ and the origin. The total space of $\mathcal O(-1)$ is therefore $\C^2$ with the origin blown up. Every non-trivial holomorphic line bundle on $\C P^1$ is holomorphically isomorphic to a positice tensor power of either $\mathcal O(-1)$ and or its dual $\mathcal O(1).$ The line bundles 
\[\mathcal O(N):=\mathcal O(1)^{\otimes N}\]
for $N\geq 1$ have an $N+1-$complex dimensional space of global sections. We denote this space by $H_{hol}^0(\C P^1,\mathcal O(N)).$ We write abusing notation $z_0$ and $z_1$ for the two global sections of $\mathcal O(1)$ that correspond to the linear functionals on $\C^2$ given by projection onto the first and second factors. Therefore,
\[H_{hol}^0(\C P^1,\mathcal O(N))=\Sym^N(z_0,z_1),\]
the space of symmetric polynomials in two complex variables. By the standard coordinate around $0=[1:0]\in \C P^1,$ we mean the coordinate $[z_0:z_1]\mapsto \frac{z_1}{z_0}=:z$ on $\C P^1\backslash \set{[0:1]}$. Relative to the frame $z_0^N:=z_0^{\otimes N},$ every holomorphic section of $\mathcal O(N)$ is represented by a complex polynomial of degree $N:$
\[\sum_{j=0}^N a_j z_1^j z_0^{N-j}=\left(\sum_{j=0}^N a_j z^j\right)\cdot z_0^N.\] 
The map $p_N(z)\mapsto p_N(z)\cdot z_0^N$ is what we mean by identifying polynomials of degree $N$ is one complex variable with holomorphic sections of $\mathcal O(N)$ ``in the usual way.''

\subsection{Scaling Limit and K\"ahler normal Coordinates} \label{S:Scaling}
\noindent A compelling argument for studying local correlations between zeros and critical points in the $N^{-1/2}-$scale normal coordinates (see Definition \ref{D:Scaled KNC}) is the following. Given any positive line bundle $(L,h)\twoheadrightarrow M$ over a complex manifold and any $\xi\in M,$ we may take $N^{-1/2}-$scale normal coordinates centered at $\xi.$ In these coordinates $(L,h)\twoheadrightarrow M$ ``converges'' to line bundle $\C\x\C^{\dim M}\twoheadrightarrow M$ with its standard K\"ahler metric $h=\frac{1}{\pi^{\dim M}}e^{-\frac{1}{2}\norm{z}^2}.$ More precisely, if $\phi$ a K\"ahler potential for $\w_h,$ then $N\phi$ is a K\"ahler potential for $\w_{h^N}=N\cdot \w_h$ and in a scaled normal coordinate centered at $\xi\in M$ for $\w_h,$ we have
\begin{equation}
  \label{E:Raison D'etre}
N\cdot \phi(z,\cl{z})=\abs{z}^2+O(N^{-1/2}).  
\end{equation}
The leading term is precisely the K\"ahler potential for $\C\x\C^{\dim M}.$ The choice of $N^{-1/2}-$scale normal coordinates therefore gives a kind of universal yardstick for studying the local correlations of zeros and critical points of random polynomials and, more generally, random section of positive line bundles. 

\section{Sz\"ego Kernels}\label{S:BS}
\noindent Suppose $p_N$ is drawn from the Hermitian Guassian ensemble corresponding to a positive smooth Hermitian metric $h$ on $\mathcal O(1).$ Viewed as a Gaussian random field, its law and hence the joint statistics of its zero and critical point processes are determined by the its covariance kernel $\Pi_N,$ the Sz\"ego Kernel associated to $(\mathcal O(N),h^N).$ Our main technical tool is therefore the $C^{\infty}$ complete asymptotic expansion for $\Pi_N$ of Shiffman and Zelditch given in \cite{OffDiag, NV}. We first recall the definition of the kernels $\Pi_N$ (Section \ref{S:Def SZ}) and introduce the related normalized Sz\"ego kernels (Section \ref{S:Normalized Szego}). Then, in Section \ref{S:Principle Bundle}, we recall the principle $S^1$ bundle $X\twoheadrightarrow \C P^1$ associated to $(\mathcal O(1),h)\twoheadrightarrow \C P^1.$ The Sz\"ego kernels are most naturally analyzed by lifting to $X.$ Finally, we use Section \ref{S:BS Asymptotics} to recall the relevant asymptotic expansions of $\Pi_N$ from \cite{OffDiag}. 

\subsection{Definition}\label{S:Def SZ}
Let $p_N=\sum_{j=0}^N a_j S_j$ be a Gaussian random polynomial (section of $\mathcal O(N)$) drawn from the Hermitian Gaussian Ensemble corresponding to a fixed smooth positive Hermitian metric on $\mathcal O(1).$ Its covariance kernel is called the Sz\"ego Kernel for $(\mathcal O(N), h^{\otimes N}):$
\begin{equation}
  \label{E:Szego Def}
\Pi_N(z,w):=\Cov(p_N(z),p_N(w))=\sum_{j=0}^N S_N^j(z)\otimes \cl{S_N^j(w)}\in H_{hol}^0(\C P^1, \mathcal O(N)\otimes \cl{\mathcal O(N)}),
\end{equation}
See the Introduction in \cite{OffDiag} for details. The family of Sz\"ego kernels $\Pi_N$ is well-understood in the general setting of a positive holomorphic line bundle $(L,h)\twoheadrightarrow M$ over a compact complex manifold $M$ (cf \cite{OffDiag, NV}). 

\subsection{Normalized Sz\"ego Kernel}\label{S:Normalized Szego} As in \cite{NV,Conditional}, it will be important to consider the Normalized Sz\"ego kernels:
\begin{equation}
  \label{E:Normalized Szego Def I}
P_N(z,w):=\frac{\norm{\Pi_N(z,w)}_{h^N}}{\sqrt{\norm{\Pi_N(z,z)}_{h^N}\norm{\Pi_N(w,w)}_{h^N}}}
\end{equation}
and
\begin{equation}
  \label{E:Normalized Szego Def I}
\twiddle{P}_N(z,w):=\frac{\norm{1\otimes \cl{\nabla^{z_0}_V} \Pi_N(z,w)}_{h^N}}{\sqrt{\norm{\Pi_N(z,z)}_{h^N}\norm{\nabla_V^{z_0}\otimes  \cl{\nabla^{z_0}_V}\Pi_N(w,w)}_{h^N}}}.
\end{equation}
We've written $V$ for an auxiliary non-vanishing local holomorphic vector field on which the value of $\twiddle{P}_N$ does not depend, and we've denoted as in Section \ref{S:Connection} by $\nabla^{z_0}$ the meromorphic connection on $\mathcal O(N)$ that extends the holomorphic derivative $\frac{d}{dz}.$ Perhaps to most natural reason to consider $P_N$ and $\twiddle{P}_N$ is probabilistic. Namely, $P_N(z,w)$ is the correlation between $p_N(z)$ and $p_N(w)$ and $\twiddle{P}_N(z,w)$ is the correlation between $p_N(z)$ and its derivative $\nabla_V^{z_0} p_N.$ 

\subsection{Princple $S^1$ Bundle}\label{S:Principle Bundle}
Consider a positive line bundle $(L,h)\twoheadrightarrow M$ over a compact K\"aher manifold and an orthonormal basis $\set{S_j}_{j=0}^{d_N}$ for $H_{hol}^0(L^N)$ with respect to the inner product (\ref{E:Inner Product}). The $N^{th}$ Sz\"ego Kernel 
\[\Pi_N(z,w)=\sum_{j=0}^{d_N} S_j(z)\otimes \cl{S_j(w)}\]
is studied in \cite{OffDiag} by lifting sections $s\in H_{hol}^0(M,L^{\otimes N})$ to $S^1$-equivariant functions on the principle $S^1$ bundle associated to $(L,h).$ More precisely, we write $h^*$ for the dual metric on the dual bundle $L^*$ and define $X\twoheadrightarrow M$ by
\[X:=\setst{v\in L^*}{\norm{v}_{h^*}=1}.\]
We denote by $\widehat{s}$ the lift of a section $s$ to the function $\widehat{s}(v):=v^{\otimes N}(s)$ on $X.$ Writing $s=f\cdot e^{\otimes N}$ for local frame $e$ of $L,$ we may write
\begin{equation}
  \label{E:Section Lift}
\widehat{s}(\theta,z):=e^{iN\theta}\norm{e(z)}_h^N\cdot f(z).
\end{equation}
Observe that 
\begin{equation}
  \label{E:Section Drop}
  \abs{\widehat{s}(\theta,z)}=\norm{s(z)}_{h^N}.
\end{equation}
The lifted Sz\"ego Kernel is then $\widehat{\Pi}_N(\alpha, z; \beta, w)=\sum_{j=0}^N \widehat{S_j}(\alpha, z)\cl{\widehat{S_j}(\beta, w)}.$ See Section 1.2 of \cite{OffDiag} for further details. In this paper, we are interested in the special case $M=\C P^1,\,L=\mathcal O(1).$ In order to study the $N^{-1/2}-$length scale behavior of $\widehat{\Pi}_N$ near a point $\xi\in \C P^1,$ we recall two definitions from Section 2.2 of \cite{NV}.
\begin{definition}\label{D:Preferred Frame}
 Fix $\xi\in \C P^1$ and $e$ a frame for $\mathcal O(1)$ in a neighborhood $U$ containing $\xi.$ The frame $e$ is called a preffered frame for $h$ at $\xi$ if
\[\norm{e(\xi)}_h=1\quad \text{and}\quad \nabla^h e(\xi)=0,\]
where $\nabla^h$ is the metric connection of $h.$
\end{definition}
\begin{definition}\label{D:Heisenberg Coords}
Fix $\xi \in \C P^1,$ a K\"ahler normal coordinate $\psi:U\gives \C$ centered at $\xi,$ and a preferred frame $e$ for $h$ at $\xi.$ A Heisenberg coordinate on $X$ centered at $\xi$ is a coordinate $\rho:S^1\x \C \gives \pi^{-1}(U)$ given by
\begin{equation}
  \label{E:Heisenberg Def}
\rho(\theta, \psi(z))=e^{i\theta}\norm{e(z)}_h e^*(z).
\end{equation}
\end{definition}

Recall from Sections \ref{S:BF} and \ref{S:Scaling} that Hermitian Gaussian ensembles have as a universal scaling limit the Bargman-Fock ensemble in $N^{-1/2}-$scale normal coordinates around any point $\xi.$ Similarly, when these ensembles are lifted to functions on $X,$ they have a universal scaling limit in Heisenberg coordinates. We refer the interested reader to Section 1.3.2 of \cite{Universality} for more details. 

\subsection{Sz\"ego and Bergman Kernel Asymptotics}\label{S:BS Asymptotics}
We now recall for the particular case of $\mathcal O(1)\twoheadrightarrow \C P^1$ the on-diagonal, near off-diagonal, and far off-diagonal asymptotics for the Sz\"ego kernels $\Pi_N$ derived in \cite{OffDiag} and \cite{NV} by Shiffman and Zelditch. We need them to prove Theorems \ref{T:Crit Two Point Function} - \ref{T:EZ}. 

\begin{Thm}[Theorem $1$ in \cite{SzegoTian}]\label{T:Szego Parametrix} There exists a $C^{\infty}$ complete asymptotic expansion:
\begin{equation}
  \label{E:Szego On-Diagonal}
\widehat{\Pi}_N(\alpha, z;\beta, z)=\frac{N}{\pi}\left(1+a_1(z)\cdot N^{-1}+\ldots\right)\cdot e^{iN(\alpha-\beta)}
\end{equation}
for certain smooth coefficients $a_j(z).$ 
\end{Thm} 

\noindent Next, we record a special case of Theorem 2.4 from \cite{NV}.
\begin{Thm}\label{T:Szego Asymptotics} In Heisenberg coordinates on $X$ around $\xi\in \C P^1,$ for $b>\sqrt{j+2k},~j,k\geq 0,$ we have the following $C^{\infty}$ asymptotic expansions:\\
{\bf 1. Far Off-Diagonal. } For $d(z,w)>b\left(\frac{\log N}{N}\right)^{1/2}$ and $j\geq 0,$ we have 
  \begin{equation}
    \label{E:Szego Far Off-Diag}
\nabla^j \widehat{\Pi}_N(\alpha, z; \beta, w)=O(N^{-k}),    
  \end{equation}
where $\nabla^j$ denotes the horizontal lift to $X$ of any $j$ mixed derivatives in $z,\cl{z},w,\cl{w}.$\\
{\bf 2. Near Off-Diagonal. } Let $\ep>0.$ In Heisenberg coordinates (see Definition \ref{D:Heisenberg Coords}) centered at $\xi,$ we have for $\abs{z}+\abs{w}<b\left(\frac{\log N}{N}\right)^{1/2}$ 
  \begin{equation}
    \label{E:Szego Near Off-Diag}
\widehat{\Pi}_N\left(\alpha, z;\beta,w\right)=e^{iN(\alpha-\beta)-z\cdot \cl{w}+\frac{1}{2}\left(\abs{z}^2+\abs{w}^2\right)}[1+R_N(z,w)],    
  \end{equation}
where
\begin{equation}\label{E:Remainder Term}
R_N(z,w)=O(N^{-1/2+\ep}),
\end{equation}
and the implied constant in equation (\ref{E:Remainder Term}) is allowed to depend on $\ep.$
\end{Thm}
\noindent Finally, we will need to recall the $C^{\infty}$ asymptotic expansions for $P_N.$
\begin{Thm}[Prop $2.6$ and $2.7$ from \cite{NV}]\label{T:Zeros Auto Corellation} Let $p_N$ be a Gaussian random polynomial defined in Section \ref{S:Def}, and consider the $N^{th}$ normalized Sz\"ego Kernel
$$P_N(z,w)=\frac{\norm{\Pi_N(z,w)}_{h^N}} {\sqrt{\norm{\Pi_N(z,z)}_{h^N}\norm{\Pi_N(w,w)}_{h^N}}}.$$
We have the following \\
{\bf 1. Far Off-Diagonal. } For $b>\sqrt{j+2k},~j,k\geq 0$ and all $\abs{z}+\abs{w}>b\left(\frac{\log N}{N}\right)^{1/2}$ we have that
\begin{equation}\label{E:Far Off-Diag Z}
\nabla^j P_N(z,w)=O(N^{-k})
\end{equation}
{\bf 2. Near Off-Diagonal. } Let $\ep,b>0$ and $\xi\in \C P^1.$ In Heisenberg coordinates centered at $\xi$ we have
\begin{equation} \label{E:Near Off-Diag Z}
P_N\left(z,w\right)=e^{-\frac{1}{2}\abs{z-w}^2}[1+R_N(z,w)],\quad \abs{\nabla^j R_N(z,w)}=O(N^{-1/2+\ep})
\end{equation}
where $\nabla^j$ denotes any $j$ interated derivatives in $z,\cl{z},w,$ or $\cl{w}$ and the implied constant is uniform in $z,w$ for $\abs{z}+\abs{w}<b\left(\frac{\log N}{N}\right)^{1/2}$ and does not depend on $\xi.$ The remainder $R_N$ satisfies in addition
\[\abs{R_N(z,w)}\leq C\abs{z-w}^2N^{-1/2+\ep},\quad \abs{\nabla R_N(z,w)}\leq \frac{C}{2}\abs{z-w}N^{-1/2+\ep}\]
for some constant $C$ uniformly for $\abs{z}+\abs{w}<\left(\frac{\log N}{N}\right)^{1/2}.$
\end{Thm}

\section{The Holomorphic Derivative $\frac{d}{dz}$ as a Meromorphic Connection on $\mathcal O(N)$}\label{S:Connection}
\noindent We use this section to study the meromorphic connection $\nabla^{z_0}$ on $\mathcal O(N)\twoheadrightarrow \C P^1,$ defined in (\ref{E:Two Connections}), that extends the euclidean derivative $\frac{d}{dz}.$ We give a formal definition in Section \ref{S:Connection Def}. We then compute various lifts of $\nabla^{z_0}$ to the principle $S^1$ bundle $X\twoheadrightarrow \C P^1$ associated to ($\mathcal O(1),h$) in Section \ref{S:Lift of Connection}. These lifts will allow us to obtain asymptotics expansions for covariant derivatives of the Sz\"ego kernels $\Pi_N$ lifted to $X$ in Section \ref{S:BS Asymptotics Derivs}.

\subsection{Definition}\label{S:Connection Def} Let $\curly P_N$ be the space of polynomials of degree at most $N$ in one complex variable. We identify $\curly P_N$ with the space of holomorphic sections $H_{hol}^0(\C P^1,\mathcal O(N))$ by pulling back along the trivialization
\[\alpha_N: \mathcal O(N)\big|_{\C P^1\backslash{\set{\infty}}}\stackrel{\cong}{\longrightarrow} \C \x \C\]
corresponding to the frame $z_0^N$ (cf Section \ref{S:CP1}). The holomorphic critical points for a holomorphic function $f$ are the zeros of the $(1,0)-$form $df=\Di{z}{f}dz.$ Interpreting $d$ as the trivial connection on $\C\x \C,$ we define the meromorphic connection  
\[\nabla^{z_0,N}:=\alpha_N^* d\]
on $\mathcal O(N).$ It is characterized by declaring the section $z_0^N$ to be parallel. The holomorphic critical points of a degree $N$ polynomial $p_N$ are therefore the same as the zeros of
\[\nabla^{z_0,N}\,(p_N\cdot z_0^N)\in H_{hol}^0(\mathcal O(N))\otimes \Omega_{mer}^{(1,0)}(\C P^1).\]
We abbreivate $\nabla^{z_0}=\nabla^{z_0,N}$ throughout and observe that $\nabla^{z_0}$ has a pole of order $1$ at infinity and is holomorphic otherwise. Indeed, writing $z=\frac{z_1}{z_0}$ and $w=z^{-1}$ for the standard coordinates around $0$ and $\infty$ on $\C P^1,$ we see that 
\begin{equation}
  \label{E:Pole at Infinity}
\nabla^{z_0} z_1^N = \nabla^{z_0} z^N\cdot z_0^N=  N\cdot z^{N-1} z_0^N\otimes dz = -\frac{N}{w}\cdot z_1^N\otimes dw.  
\end{equation}
Since the probability that a guassian random polynomial $p_N$ vanishes at infinity is $0,$ $\nabla^{z_0} p_N$ has a simple pole at infinity almost surely. 

\subsection{Lift of $\nabla^{z_0}$ to Principle $S^1$ Bundle}\label{S:Lift of Connection}
In this section, we compute various lifts of $\nabla^{z_0}$ to the principle $S^1$ bundle $X\twoheadrightarrow \C P^1$ associated to a fixed smooth Hermitian metric $h$ on $\mathcal O(1)$ (cf Section \ref{S:Principle Bundle}). We will denote by $\widehat{\nabla^{z_0}}$ the lift of $\nabla^{z_0}$ to $X$ and by $\widehat{\nabla^{z_0}}^{(1,0)}$ and $\widehat{\nabla^{z_0}}^{(0,1)}$ its $(1,0)$ and $(0,1)$ parts. 

We continue to write $\phi_{z_0}:\C P^1\backslash{\set{\infty}}\gives \mathbb R$ for the K\"ahler potential 
\[\phi_{z_0}(z)=\log \norm{z_0(z)}_h^{-2}\]
over $\C P^1\backslash\set{\infty}$ for $\w_h.$ For $\xi\in \C P^1\backslash{\set{\infty}}$ and any holomorphic coordinate $z$ centered at $\xi,$ we write $\gamma_0$ for the ``leading harmonic part'' of $\phi_{z_0}:$
\[\gamma_0(z,\cl{z}):=\phi_{z_0}(\xi)+\Di{z}{\phi_{z_0}}(\xi)\cdot z+\Di{\cl{z}}{\phi_{z_0}}(\xi)\cdot \cl{z}+\frac{1}{2}\left[\DDi{z}{\phi_{z_0}}\cdot z^2+\DDi{\cl{z}}{\phi_{z_0}}\cdot \cl{z}^2\right].\]
Similarly, writing as in Section \ref{S:CP1} $z_1$ for the standard frame of $\mathcal O(1)$ over $\C P^1\backslash{\set{0}},$ we define $\phi_{z_1}:=\log \norm{z_1}_h^{-2}$ and introduce
\[\gamma_1(z,\cl{z}):=\phi_{z_1}(\xi)+\Di{z}{\phi_{z_1}}(\xi)\cdot z+\Di{\cl{z}}{\phi_{z_1}}(\xi)\cdot \cl{z}+\frac{1}{2}\left[\DDi{z}{\phi_{z_1}}\cdot z^2+\DDi{\cl{z}}{\phi_{z_1}}\cdot \cl{z}^2\right].\]
\begin{Lem}[Lift of $\nabla^{z_0}$ in Heisenberg Coordinates]\label{L:Connection Lift}
Fix $\xi \in \C P^1$ and a Heisenberg coordinate on $X$ centered at $\xi.$ If $\xi\neq \infty,$ then we may write
\begin{equation}
  \label{E:Lifted Cnx}
\widehat{\nabla^{z_0}}^{(1,0)}(\alpha,z)= \partial_z+ \frac{N}{2}\cdot \partial_z\left[\phi_{z_0}(z)+\gamma_0(z)\right].  
\end{equation}
Further, fix Heisenberg coordinates centered at $\xi$ on the diagonal of $X\x X.$ The (differential) order $0$ part of the lift of $\nabla^{z_0}\otimes \cl{\nabla^{z_0}}$ is
\begin{equation}
  \label{E:Lift on Diag}
N\partial_z\cl{\partial}_{z}\phi_{z_0}(z)+N^2\partial_z \phi_{z_0}\cdot \cl{\partial}_z\phi_{z_0}(z).
\end{equation}
The order $1$ and $2$ parts are $O(N).$ Finally, if $\xi=\infty$ and we write $\psi(z)=w$ for the change of coordinates to the usual holomorphic coordinate $w=\frac{z_0}{z_1}$ at $\xi,$ we have
\begin{equation}
  \label{E:Lifted Cnx Infinity}
\widehat{\nabla^{z_0}}^{(1,0)}(\alpha,z)= \partial_z+ \frac{N}{2}\cdot \partial_z\left[\phi_{z_1}(z)+\gamma_1(z)+\log \psi(z)^2\right].  
\end{equation}
In Heisenberg coordinates centered at $\infty$ on the diagonal of $X\x X,$ the lift of the (differential) order $0$ part of $\nabla^{z_0}\otimes \cl{\nabla^{z_0}}$ is
\begin{equation}
  \label{E:Lift on Diag Infinity}
\frac{N^2}{\abs{\psi(z)}^2}+N\Re\left(\frac{\partial_z\phi_{z_1}(\psi(z))}{\psi(z)}\right)+N\partial_z\cl{\partial}_{z}\phi_{z_0}(\psi(z))+N^2\partial_z \phi_{z_0}(\psi(z))\cdot \cl{\partial}_z\phi_{z_0}(\psi(z)).
\end{equation}
The order $1$ and $2$ parts are $O(N).$
 \end{Lem}

\noindent Taylor expanding the results of Lemma \ref{L:Connection Lift}, gives the following corollary.

\begin{corollary}[Lift of $\nabla^{z_0}$ in Scaled Heisenberg Coordinates]\label{C:Local Connection Lift}
Fix $\xi$ and a scaled Heisenberg coordinate centered at $\xi.$ In the notation of Lemma \ref{L:Connection Lift}, the lifted connection $\widehat{\nabla^{z_0}}$ exhibits three different behaviors. \\
{\bf Case 1 $d\phi_{z_0}(\xi)\neq 0$.} If $d\phi_{z_0(\xi)}\neq 0$ or $\xi\neq \infty,$ then we have
 \begin{equation}
  \label{E:Lift at NonCrit}
  \widehat{\nabla^{z_0}}^{(1,0)}(u, \theta)= \Di{u}{\phi_{z_0}}\bigg|_{\xi}\otimes N^{1/2}\cdot (1+O(N^{-1/2}))du
\end{equation}
and the (differential) order $0$ part of the $(1,1)$ part of $\nabla^{z_0}\otimes \cl{\nabla^{z_0}}$ lifted to the diagonal is
\begin{equation}
  \label{E:Lift to Diag at NonCrit}
\abs{\Di{u}{\phi_{z_0}}\bigg|_{\xi}}^2\otimes N\cdot (1+O(N^{-1/2}))\cdot du\wedge d\cl{u}.
\end{equation}
{\bf Case 2 $d\phi_{z_0}(\xi)=0$.} If $d\phi_{z_0}(\xi)=0$ and $\xi\neq \infty$ then 
\begin{equation}
  \label{E:Lift at Crit}
\widehat{\nabla^{z_0}}^{(1,0)}(u,\theta)=\partial_u+\frac{1}{2}\left(\DDi{u}{\phi_{z_0}}\bigg|_{\xi}\cdot u+\cl{u}\right)\otimes du+O(N^{-1/2}),
\end{equation}
and the (differential) order $0$ part of the $(1,1)$ part of $\nabla^{z_0}\otimes \cl{\nabla^{z_0}}$ lifted to the diagonal is
\begin{equation}
  \label{E:Lift to Diag at Crit}
 \left(1+\abs{\DDi{u}{\phi_{z_0}}\bigg|_{\xi}\cdot u +\cl{u}}^2+O(N^{-1/2})\right)\otimes du\wedge d\cl{u}.
\end{equation}
{\bf Case 3 $\xi=\infty$. } Finally, in the case $\xi=\infty,$ we have
\begin{equation}
  \label{E:Lift at Infinity}
\widehat{\nabla^{z_0}}^{(1,0)}(u,\theta)=-\Di{w}{\psi}\bigg|_{\xi}\cdot\frac{1}{u}\otimes N(1+O(N^{-1/2})) du,
\end{equation}
and the (differential) order $0$ part of the $(1,1)$ part of $\nabla^{z_0}\otimes \cl{\nabla^{z_0}}$ lifted to the diagonal is 
\begin{equation}
  \label{E:Lift to Diag at Infinity}
\abs{u}^{-2}\otimes N^3(1+O(N^{-1/2}))\cdot du\wedge d\cl{u}.
\end{equation}
\end{corollary}

\begin{remark}
Equation (\ref{E:Lift at NonCrit}) shows that in scaled Heisenberg coordinates centered at $\xi$ satisfying $d\phi_{z_0}(\xi)\neq 0$ or $\xi=\infty,$ $\widehat{\nabla^{z_0}}$ is an order $0$ operator to leading order in $N.$ This explains analytically why, in the large $N$ limit, zeros and critical points are indistinguishable in this case. In contrast, (\ref{E:Lift at Crit}) shows that $\widehat{\nabla^{z_0}}$ is an order $1$ operator to leading order in $N$ if $d\phi_{z_0}(\xi)=0.$ 
\end{remark}

\begin{proof}[Proof of Lemma \ref{L:Connection Lift}]
First suppose that $\xi\in \C P^1\backslash{\set{\infty}}.$ We begin by constructing Heisenberg coordinates on $X$ centered at $\xi.$ With $\twiddle{\gamma}_0$ denoting the harmonic conjugate of $\gamma_0$, we observe that the frame
\begin{equation}
  \label{E:Preferred Frame}
e_L:=e^{\frac{1}{2}\left(\gamma_0+i\twiddle{\gamma}_0\right)}\cdot z_0  
\end{equation}
is a preffered frame near $\xi$ in the sense of Definition \ref{D:Preferred Frame}. Combined with any K\"ahler normal coordinate centered at $\xi,$ the frame $e_L$ allows us to construct Heisenberg coordinates centered at $\xi.$ Note that
\[\norm{e_L}_h^N=e^{\frac{N}{2}\left(\gamma_0-\phi_{z_0}\right)}.\]
Fix $S\in H_{hol}^0(\C P^1, \mathcal O(N)),$ and write $S=f\cdot e_L$ locally. Since 
\[\Di{z}{\gamma_0}=i\Di{z}{\twiddle{\gamma_0}},\]
we use expression (\ref{E:Section Lift}) for lifting sections of $\mathcal O(N)$ to functions on $X$ to write  
\[\widehat{\nabla_{\Di{z}{}}^{z_0}S}(\alpha,z)=\left(\Di{z}{f}(z)+Nf(z)\Di{z}{\gamma_0}(z)\right)e^{\frac{N}{2}(\gamma_0(z)-\phi_{z_0}(z))}e^{iN\theta}.\]
Therefore, the $(1,0)$ part of the lift of the connection $\nabla^{z_0}$ on $\mathcal O(N)$ to the trivial line bundle $X\x \C$ is
\[\widehat{\nabla^{z_0}}^{(1,0)}(z)=e^{\frac{N}{2}(\gamma_0(z)-\phi_{z_0}(z))}\circ \left(\partial_z +N\partial_z \gamma_0(z)\right)\circ e^{-\frac{N}{2}(\gamma_0(z)-\phi_{z_0}(z))}= \partial_z +\frac{N}{2}\partial_z\left(\phi_{z_0}(z)+\gamma_0 (z)\right),\]
where $\circ$ denotes composition of differential operators. This confirms (\ref{E:Lifted Cnx}). Next, to deduce (\ref{E:Lift on Diag}), we use the definition of Heisenberg coordinates (\ref{E:Heisenberg Def}) to write 
\[\widehat{S\otimes \cl{S}}(\alpha, z; \beta, z)=\abs{f(z)}^2e^{N\left(\gamma(z)-\phi_{z_0}(z)\right)}e^{iN(\alpha-\beta)}.\]
Similarly, the lift of $\nabla_{\Di{z}{}}^{z_0} S(z)\otimes \cl{\nabla_{\Di{z}{}}^{z_0} S(z)}$ to $X\x X:$ 
\begin{align*}
\abs{\Di{z}{f}(z)+Nf\Di{z}{\gamma}(z)}^2e^{N\left(\gamma(z)-\phi_{z_0}(z)\right)}e^{iN(\alpha-\beta)}.
\end{align*}
Comparing the two previous expressions, we see that the $(1,1)$ part of the lift of $\nabla^{z_0}\otimes \cl{\nabla^{z_0}}$ to the diagonal of $X\x X$ is given by 
\begin{equation}
  \label{E:Lift on Diago Conjugation}
\left[e^{N\left(\gamma_0(z,w)-\phi_{z_0}(z,w)\right)}\circ \left(\partial_z +N\partial_z \gamma_0(z)\right)\circ \left(\cl{\partial}_w +N\cl{\partial}_w \gamma_0(w)\right)\circ e^{-N\left(\gamma_0(z,w)-\phi_{z_0}(z,w)\right)}\right]\bigg|_{z=w}.  
\end{equation}
Here $\gamma(z,w)$ and $\phi_{z_0}(z,w)$ denote the extensions of $\gamma$ and $\phi_{z_0}$ from the diagonal of $X\x X$ that are holomorphic in $z$ and anti-holomorphic in $w.$ The expression (\ref{E:Lift on Diago Conjugation}) may be written as 
\[\underbrace{N\partial_z\cl{\partial}_{z}\phi_{z_0}+N^2\partial_z \phi_{z_0}\cdot \cl{\partial}_z\phi_{z_0}}_{\text{order 0}}+\underbrace{N\partial_z\phi_{z_0}\cl{\partial}_z+N\cl{\partial}_z\phi_{z_0}\partial_z}_{\text{order 1}}+\underbrace{\partial_z\cl{\partial}_z}_{\text{order 2}},\]
confirming (\ref{E:Lift on Diag}).

Finally, we consider the case when $\xi=\infty.$ This case needs to be treated separately since no parallel frame for $\nabla^{z_0}$ exists near infinity. We write $z_1$ for the usual frame of $\mathcal O(1)$ over $\C P^-\set{0}$ (cf Section \ref{S:CP1}). As before, 
\[e_L:=e^{\frac{1}{2}\left(\gamma_1+i\twiddle{\gamma}_1\right)}z_1\]
is a preferred frame for $\mathcal O(1)$ near $\xi=\infty.$ Recall from (\ref{E:Pole at Infinity}) that if we denote by $w$ the standard coordinate around $\infty,$
\[\nabla^{z_0} z_1^N=-\frac{N}{w}z_1^N\otimes dw.\]
Write $z=\psi(w)$ for the change of coordinates to a K\"ahler normal coordinate for $\w_h$ centered at $\xi.$ Note that $\Di{w}{\gamma_1}=i\Di{w}{\twiddle{\gamma}_1}.$ For $S=f\cdot e_L \in H_{hol}^0(\C P^1, \mathcal O(N))$ as before
\[\widehat{\nabla_{\Di{z}{}}^{z_0}S}(\theta,z)=\left(\Di{z}{f}(z)+Nf(z)\left[\Di{z}{\gamma_1}(z)+\frac{1}{\psi(z)}\Di{z}{\psi}(z)\right]\right)e^{\frac{N}{2}(\gamma_1(z)-\phi_{z_1}(z))}e^{iN\theta}.\]
The $(1,0)$ part of the lift of the connection $\nabla^{z_0}$ on $\mathcal O(N)$ to the trivial line bundle $X\x \C$ is therefore
\[\widehat{\nabla^{z_0}}^{(1,0)}(\alpha,z)= \partial_z+ \frac{N}{2}\cdot \partial_z\left[\phi_{z_1}(z)+\gamma_1(z)+\log \psi(z)^2\right],\]  
confirming (\ref{E:Lifted Cnx Infinity}). Equation (\ref{E:Lift on Diag Infinity}) is derived exactly like (\ref{E:Lift on Diag}). 
\end{proof}

\section{Asymptotics of $\nabla^{z_0}$ Derivatives of the Szego Kernel}\label{S:BS Asymptotics Derivs}
\noindent We now combine the formulas from Section \ref{S:Lift of Connection} for the different lifts of $\nabla^{z_0}$ to $X$ with the asymptotics of Theorems \ref{T:Szego Asymptotics} and \ref{T:Zeros Auto Corellation} to derive asymptotic expasions for the covariant derivatives of $\widehat{\Pi}_N$ and $P_N$ with respect to $\nabla^{z_0}.$  
\begin{Thm}\label{T:Szego Parametrix Derivs} There exists a $C^{\infty}$ complete asymptotic expansion for lift of $\nabla^{z_0}\otimes \cl{\nabla^{z_0}} \Pi_N$ to the diagonal in $X\x X$:
\begin{equation}
  \label{E:Szego On-Diagonal}
\left[N\frac{\partial^2\phi_{z_0}}{\partial z \partial\cl{z}}+N^2\abs{\Di{z}{\phi_{z_0}}(z)}^2+O(1)\right]\otimes dz\wedge d\cl{z}.
\end{equation}
\end{Thm} 
\begin{proof}
Equation (\ref{E:Szego On-Diagonal}) follows immediately from the asymptotic expansion (\ref{E:Szego On-Diagonal}) of Theorem \ref{T:Szego Parametrix} and the lift of $\nabla^{z_0}\otimes \cl{\nabla^{z_0}}$ to the diagonal of $X\x X$ given in (\ref{E:Lift on Diag}) of Lemma \ref{L:Connection Lift}.
\end{proof}

\noindent We now use the asymptotic expansions for $\widehat{\Pi}_N(z,w)$ given in Theorem \ref{T:Szego Asymptotics} combined with Lemma \ref{L:Connection Lift} and Corollary \ref{C:Local Connection Lift} to control $\twiddle{P}_N(z,w),$ the correlation between $p_N(z)$ and $\nabla^{z_0}p_N(w).$ 

\begin{Thm}\label{T:Zeros and Crits Corellation} We use the notation of Theorem \ref{T:Szego Asymptotics} and consider 
$$\twiddle{P}_N(z,w)=\frac{\norm{1\otimes \cl{\nabla_{V}^{z_0}} \Pi_N(z,w)}_{h^N}}{\left[\norm{\Pi_N(z,z)}_{h^N}\norm{\nabla_V^{z_0}\otimes \cl{\nabla_{V}^{z_0}} \Pi_N(w,w)}_{h^N}\right]^{1/2}}.$$
{\bf 1. Far Off-Diagonal Asymptotics. } For $b>\sqrt{j+2k},~j,k\geq 0,$ we have
\begin{equation}\label{E:Far Off-Diag}
\nabla^j \twiddle{P}_N(z,w)=O(N^{-k})
\end{equation}
uniformly for $\abs{z-w}\geq b\cdot \left(\frac{\log N}{N}\right)^{1/2}.$\\
{\bf 2. Near Off-Diagonal Asymptotics. } Let $\ep>0$ and $\xi\in \C P^1.$ Take a $N^{-1/2}$-scale normal coordinate centered at $\xi$ and write $\phi_{z_0}:=\log \norm{z_0}_h^{-2}.$ If $d\phi_{z_0}(\xi)=0,$ then 
\begin{equation} \label{E:Near Off-Diag}
\twiddle{P}_N(z,w)=\frac{\abs{\DDi{\cl{w}}{\phi_{z_0}}\big|_{\xi}\cdot\cl{w}+z}}{\sqrt{1+\abs{\DDi{\cl{w}}{\phi_{z_0}}\big|_{\xi}\cdot\cl{w}+w}^2}}e^{-\frac{1}{2}\abs{z-w}^2}[1+\twiddle{R}_N(z,w)].
\end{equation}
If $d\phi_{z_0}(\xi)\neq 0$ or $\xi=\infty,$ then 
\begin{equation} \label{E:Near Off-Diag Infinity}
\twiddle{P}_N(z,w)=e^{-\frac{1}{2}\abs{z-w}^2}[1+\twiddle{R}_N(z,w)].
\end{equation} The remainders $\twiddle{R}_N$ in (\ref{E:Near Off-Diag}) and (\ref{E:Near Off-Diag Infinity}) satisfy $\abs{\nabla^j \twiddle{R}_N(z,w)}=O(N^{-1/2+\ep})$ as well as
\begin{equation}
  \label{E:Remainder Estimate Derivs}
\abs{\twiddle{R}_N(z,w)}\leq C\abs{z-w}N^{-1/2+\ep}  
\end{equation}
for some constant $C$ uniformly for $\abs{z}+\abs{w}<\left(\frac{\log N}{N}\right)^{1/2}.$ 
\end{Thm}
\begin{proof}
Equations (\ref{E:Far Off-Diag}) and (\ref{E:Near Off-Diag}) follow immediately by differentiating (\ref{E:Far Off-Diag}) and (\ref{E:Near Off-Diag}) by the lift of $\nabla^{z_0}$ to $X$ given in (\ref{E:Lift at NonCrit}) and (\ref{E:Lift at Crit}).
\end{proof}
\noindent We conclude by recording some estimates on $\twiddle{P}_N$ that we will be useful in proving Theorems \ref{T:Crit Two Point Function} and \ref{T:Sendov}.
\begin{corollary}\label{C:Normalized Estimates}
Fix $\xi\in \C P^1$ and a $N^{-1/2}-$scale K\"ahler normal coordinate centered at $\xi.$ If $d\phi_{z_0}(\xi)=0,$ then
\begin{equation}
  \label{E:Normalized Estimate Compat}
\twiddle{P}_N(z,w)\leq C <1  
\end{equation}
uniformly in $N$ for some universal constant $C$ for all $\abs{z-w}\leq \log N.$ If $d\phi_{z_0}(\xi)\neq 0$ or $\xi=\infty,$ then 
\begin{equation}
  \label{E:Normalized Estimate Non-Compat}
\twiddle{P}_N(z,w)\leq 1-K\cdot \frac{1}{\sqrt{N}}  
\end{equation}
uniformly in N for some constant $K>0$ and for all $\abs{z-w}\leq \log N.$ If we also assume that $\abs{z-w}=O(N^{-1/4}),$ then for a constant $R>0$
\begin{align}
\label{E:P Est 1}  P_N^{\xi}(z,w)^2 &=1-\abs{z-w}^2-R\cdot N^{-1/2} +O(N^{-3/4+\ep})\\
\label{E:P Est 2}  \partial_z \left(P_N^{\xi}(z,w)^2\right) &=O(N^{-1/4})\\
\label{E:P Est 3}  \partial_z\cl{\partial}_z P_N^{\xi}(z,w)^2 &=1+O(N^{-1/2+\ep}).
\end{align}
\end{corollary}
\begin{proof}  
Fix $\xi \in \C P^1$ and a $N^{-1/2}-$scale K\"ahler normal coordinate centered at $\xi.$ Suppose first that $d\phi_{z_0}(\xi)=0.$ Equation (\ref{E:Near Off-Diag}) implies
\[\twiddle{P}_N(z,w) =\frac{\abs{\alpha\cl{w}+z}} {\sqrt{1+ \abs{\alpha\cdot\cl{w}+w}^2+O(N^{-1})}}\cdot e^{-\frac{1}{2}\abs{z-w}^2}(1+O(N^{-1/2})),\]
where we've denoted $\alpha=\DDi{\cl{w}}{\phi_{z_0}}(\xi).$ Setting $\eta = \alpha\cl{w}+w$ and $\xi = z-w,$ we have
\[\twiddle{P}_N^2(z, w)=\frac{\abs{\eta+\xi}^2}{1+\abs{\eta}^2}e^{-\abs{\xi}^2}.\]
When $\abs{\xi}$ is large, (\ref{E:Normalized Estimate Compat}) is satisfied. For $\abs{\xi}$ bounded above, we taylor expand $e^{-\abs{\xi}^2}$ to write
\[\twiddle{P}_N^2(z, w) \leq \frac{\abs{\eta +\xi}^2}{(1+\abs{\eta}^2)(1+\abs{\xi}^2+\abs{\xi}^4/2)}+O(N^{-1/2}).\]
Estimating the numerator above by $\abs{\xi}^2+2\abs{\xi}\abs{\eta}+\abs{\eta}^2$ and using that $1+\abs{\xi}^2\abs{\eta}^2\geq 2\abs{\xi}\abs{\eta}$ confirms (\ref{E:Normalized Estimate Compat}). 

Suppose next that $d\phi_{z_0}(\xi)\neq 0$ or $\xi=\infty.$ The asymptotic expansion (\ref{E:Far Off-Diag}) yields
\begin{equation}
  \label{E:Normalized Non-Compat Beta}
\twiddle{P}_N(z,w)=\frac{\beta} {\sqrt{\beta^2+O(N^{-1/2})}}\cdot e^{-\frac{1}{2}\abs{z-w}^2}\left(1+O(N^{-1/2})\right),  
\end{equation}
where we set $\beta:=\abs{\Di{z}{\phi_{z_0}}(\xi)}.$ Estimating $e^{-\frac{1}{2}\abs{z-w}^2}\leq 1$ we see that
\[P_N\left(z, w\right)\leq 1-\frac{K}{\sqrt{N}},\]
for some $K>0,$ as desired. Finally, assume additionally that $\abs{z-w}=O(N^{-1/4}).$ Writing $e^{-\abs{z-w}^2}=1-\abs{z-w}^2+O(N^{-1})$ in (\ref{E:Normalized Non-Compat Beta}), we conclude (\ref{E:P Est 1}). Equations (\ref{E:P Est 2}) and (\ref{E:P Est 3}) now easily follow from (\ref{E:P Est 1}) and the remainder estimate (\ref{E:Remainder Estimate Derivs}).
\end{proof}

\section{Relation of Sz\"ego Kernels to Zeros and Critical Points}\label{S:Relation}
\noindent In this section, we give explicit formulas for $\E{Z_{p_N}}$ and $\E{C_{p_N}}$ in terms of the Sz\"ego Kernels $\Pi_N$ and the connection $\nabla^{z_0}.$ Lemma \ref{L:Zero Expected Density} is a rather general and simple result that was proved in various guises in \cite{PLL, EK, Equil} and essentially in the present form as Proposition 2.1 in \cite{NV}. Both its conclusion and the ideas in its proof will be used throughout.

\begin{Lem}[Probabilistic Poincare-Lelong Formula] \label{L:Zero Expected Density} \,\,Let $M$ be a complex manifold without boundary and $L\twoheadrightarrow M$ be a holomorphic line bundle endowed with a positive Hermitian metric $h.$ Let 
$$\set{\Sigma_j,~j=1,\ldots,J}\subset H_{mer}^0(M,L)$$
be arbitrary merormorphic sections that are not all identically zero. Define a Gaussian random section $s$ by
$$s(z):=\sum_{j=1}^J a_j\Sigma_j(z),~~~~a_j\sim N(0,1)_{\C}~~i.i.d.$$
Denoting by $\E{\cdot}$ the expected value operator for the standard complex Gaussian vector $[a_1,\ldots,a_J]$ and by $Z_s$ and $P_s$ the currents of integration over the zeros and poles of $s,$ we have
\begin{equation}
  \label{E:General Expected Zeros Global}
\E{Z_s-P_s}(z)=\frac{i}{2\pi}\partial \cl{\partial}\log \norm{\Pi(z,z)}_h+ \w_h(z)
\end{equation}
Here $\w_h$ is first chern class of $(L,h)$ and $\Pi(z,w):=\sum_{j=1}^N \Sigma_j(z)\otimes \cl{\Sigma_j(w)}$ is the associated Sz\"ego kernel. In a local holomorphic frame $e_L$ of $L,$ we write $\Sigma_j=\sigma_j\cdot e_L$ and obtain the following equivalent expression:
\begin{equation}
  \label{E:General Expected Zeros Local}
\E{Z_s-P_s}(z)=\frac{i}{2\pi}\partial \cl{\partial}\log \sum_{j=1}^J \abs{\sigma_j(z)}^2.
\end{equation}
\end{Lem}
\noindent Lemma \ref{L:Zero Expected Density} is a probabilitist analog of the following well-known result:
\begin{Lem}[Poincare-Lelong Formula]\label{L:PLL}
  Let $L\twoheadrightarrow M$ be a holmorphic line bundle over a complex manifold and suppose $s\in H_{hol}^0(M,L)$ is a merormophic section. Write $Z_s$ and $P_s$ for the currents of integration over the zeros and poles of $s$ and express $s=f\cdot e$ relative to a local frame $e$ of $L.$ Then 
\[Z_s-P_s=\frac{i}{\pi}\partial \cl{\partial}\log \abs{f}.\]
\end{Lem}

\begin{proof}[Proof of Lemma \ref{L:Zero Expected Density}]
To prove (\ref{E:General Expected Zeros Global}), it is enough to verify (\ref{E:General Expected Zeros Local}) as for any local frame $e_L$ of $L,$ $\w_h$ is given locally by $\frac{i}{2\pi}\partial \cl{\partial}\log \norm{e_L}_h^{-2},$ making 
\[\frac{i}{2\pi}\partial \cl{\partial}\log \norm{\Pi(z,z)}\cdot \norm{e_L}_h^{-2}=\frac{i}{2\pi}\partial \cl{\partial}\log\sum_{j=1}^J \abs{\sigma_j(z)}^2\]
the local expression for
$$\frac{i}{2\pi}\partial \cl{\partial}\log \norm{\Pi_N(z,z)}_h + \w_h(z).$$
We will abbreviate $s=\inprod{a}{\Sigma}=\inprod{a}{\sigma}\cdot e_L,$ where
$$a=[a_0,\ldots,a_J],~~\Sigma=[\Sigma_0,\ldots,\Sigma_J],~~\sigma=[\sigma_0,\ldots,\sigma_J].$$
For any smooth test function $\psi,$ we apply the Poincare-Lelong formula to write
\begin{align*}
\left(\E{Z_s-P_s},\psi\right) &= \frac{i}{\pi}\E{\left(\log \abs{\inprod{a}{\sigma}},\partial \cl{\partial} \psi\right)}\\
                              &= \frac{i}{2\pi}\left(\log \sum_{j=0}^N \abs{\sigma_j}^2,\partial \cl{\partial} \psi\right)+\frac{i}{\pi}\E{\left(\log \abs{\inprod{a}{u}},\partial \cl{\partial} \psi\right)}.
\end{align*}
We've set 
$$u(z)=\frac{\sigma(z)}{\sum_{j=0}^N \abs{\sigma_j}^2},$$
a unit vector at all but finitely many points. The second term vanishes due to the unitary invariance of guassian measure. Indeed, as in Section 3.2 of \cite{Quantum}, we write the second term as
\begin{align*}
\frac{i}{\pi}\int_{\C^{N+1}}\left[ \int_{\C P^1}\log \abs{\inprod{a}{u(z)}}\partial_z \cl{\partial}_z \psi(w) dz d\cl{z}\right]d\gamma(a),
\end{align*}
where $d\gamma(a)=\frac{1}{\pi^{J+1} }e^{-\norm{a}^2}da$ is the Gaussian density on $\C^{J+1}$. It is straight-forward to check that the integrand is in $L^1,$ allowing us to change the order of integration. For almost every $z\in \C P^1,$ we have that $\inprod{a}{u(z)}\stackrel{\curly D}{=}a_1,$ a standard normal random variable on $\C.$ The integral 
\begin{align*}
\int_{\C^{N+1}}\log \abs{\inprod{a}{u(z)}} d\gamma(a)
\end{align*}
is therefore a universal constant independent of $z$ and is killed by the operator $\partial_z \cl{\partial}_z.$
\end{proof}

\noindent We have the following
\begin{corollary}\label{C:Crit Expected Density} Let $p_N$ be a degree $N$ polynomial drawn from the Hermitian Gaussian Ensemble corresponding to a smooth positive Hermitian metric $h$ on $\mathcal O(1).$ Let $\w_h$ denote the first chern class of $(\mathcal O(1),h).$ Write $C_{p_N}$ for the current of integration over the critical point set of $p_N.$ For any $\xi \in \C P^1$ and any non-vanishing holomorphic vector field $V$ in a neighbhorhood of $\xi,$ we have 
\begin{equation}
  \label{E:Crit Expected Value I}
\E{C_{p_N}}(\xi)=\frac{i}{2\pi}\partial \cl{\partial} \log \left[\norm{\nabla_V^{z_0}\otimes\nabla_{\cl{V}}^{z_0}\Pi_N(\xi,\xi)}_h\right]+ N\cdot \w_h(\xi) + \delta_{\infty}(\xi) .
\end{equation}
Consequently, $\E{C_{p_N}}$ is a smooth $(1,1)-$form.
\end{corollary}
\begin{proof}
By definition, we may write locally
\[C_{p_N}=Z_{\nabla_V^{z_0} p_N},\]
where $V$ is any non-vanishing holomorphic vector field. Recall from Section \ref{S:Connection} that $\nabla_V^{z_0} p_N$ is holomorphic except at $\infty$ where it has a simple pole almost surely. Therefore, denoting by $P_s$ the current of integration over the poles for a section $s\in H_{mer}^0(\C P^1,\mathcal O(N)),$ we have 
\[\E{P_{\nabla_V^{z_0} p_N}}=\delta_{\infty}.\]
Combining the Poincare-Lelong formula with Lemma \ref{L:Zero Expected Density} applied to $M=\C P^1,$ $L=(\mathcal O(N),h)$,  $\w=w_{h^N}$ and the Gaussian random section $p_N$ proves (\ref{E:Crit Expected Value I}). That $\E{C_{p_N}}$ is a smooth $(1,1)-$form  away from infinity is clear from (\ref{E:Crit Expected Value I}). To check that $\E{C_{p_N}}$ is smooth at $\infty$ we choose $\set{S_j}_{j=0}^N,$ an orthonormal basis of $H_{hol}^0(\C P^1,\mathcal O(N))$ with respect to the inner product (\ref{E:Inner Product}) and compute in the standard holomorphic coorinate centered at $\infty.$ Relative to the usual frame $z_1$ of $\mathcal O(1)$ over $\C P^1\backslash{\set{0}},$ we may write $S_j=f_j\cdot z_1^N.$ By (\ref{E:Pole at Infinity}),
\[\norm{\nabla^{z_0}\otimes \cl{\nabla^{z_0}}\Pi_N(z,w)}_{h^N}=N^2\frac{\norm{z_1}_h^{2N}}{\abs{w}^2}\cdot \sum_{j=0}^N \abs{\Di{w}{f_j}(w)}^2.\]
Note that $\sum_{j=0}^N \abs{\Di{w}{f_j}(w)}^2>0.$ Writing 
\[\delta_{\infty}(w)=\frac{i}{\pi}\partial_w \cl{\partial}_w\log \abs{w}^2\quad \text{and}\quad N\cdot \w_h(w)=\frac{i}{2\pi}\partial_w\cl{\partial}_w\log \norm{z_1(w)}_h^{-2N}\]
in (\ref{E:Crit Expected Value I}) shows that $\E{C_{p_N}}$ is smooth at infinity. 
\end{proof}

\section{Expected Density of Critical Points: Proof of Theorem \ref{T:EZ}}\label{S:Proof 1pt} 
\noindent Throughout, we denote by $p_N$ a degree $N$ polynomial drawn from the Hermitian Gaussian Ensemble corresponding to a smooth positive Hermitian metric $h$ on the line bundle $\mathcal O(1)\twoheadrightarrow \C P^1$ (see Section \ref{S:Def}). We continue to write $\nabla^{z_0}$ for both the meromorphic connection on $\mathcal O(1)$ that extends the euclidean derivative (see Section \ref{S:Connection}) and the connections $\left(\nabla^{z_0}\right)^{\otimes N}$ it induces on $\mathcal O(N)\twoheadrightarrow \C P^1.$ 
\begin{proof}[Proof of Theorem \ref{T:EZ}]
Since $\E{C_{p_N}}$ is smooth by Corollary \ref{C:Crit Expected Density}, to verify the asymptotics for $\frac{1}{N-1}\E{C_{p_N}}$ given in equation (\ref{E:Crit Global One Point}) it suffices to show that for each $\psi \in C_c^2(\C P^1\backslash{\set{\infty}})$ we have
\begin{equation}
\label{E:Error}
\int_{\C P^1\backslash{\set{\infty}}} \psi(z)\E{C_{p_N}}(z)= \int_{\C P^1\backslash{\set{\infty}}} \psi(z)\cdot N \w_h(z) +O(1)
\end{equation}
with the implied constant independent of $N.$ We write $z$ for the standard holomorphic coordinate on $\C P^1\backslash{\set{\infty}}$ and $z_0$ for the standard frame of $\mathcal O(1)$ over $\C P^1\backslash{\set{\infty}}$ as in Section \ref{S:CP1}. From equation (\ref{E:Lift on Diag}) and the asymptotic expansion (\ref{E:Szego On-Diagonal}), we conclude that
\[\norm{\nabla_{\Di{z}{}}^{z_0}\otimes \cl{\nabla_{\Di{z}{}}^{z_0}}\Pi_N(z,z)}_{h^N}=\abs{\frac{\partial^2}{\partial z\partial \cl{z}}\phi_{z_0}(z) + N\abs{\Di{z}{\phi_{z_0}}(z)}^2+O(N^{-1})}.\]
The assumption that $h$ is a postive metric means that $\frac{\partial^2}{\partial z\partial \cl{z}}\phi_{z_0}(z)>0.$ So we may omit the absolute values in the previous line and use (\ref{E:Crit Expected Value I}) of Corollary \ref{C:Crit Expected Density} to write
\begin{equation}
  \label{E:Error Gen}
\E{C_{p_N}}(z)-N\cdot \w_h(z)=\frac{i}{2\pi}\partial \cl{\partial} \log \left[\frac{\partial^2}{\partial z\partial \cl{z}}\phi_{z_0}(z) + N\abs{\Di{z}{\phi_{z_0}}(z)}^2\right]+O(N^{-1}).
\end{equation}
Hence, we seek to show that
\begin{equation}
\label{E:Error II}
\int_{\C P^1\backslash{\set{\infty}}} \frac{i}{2\pi}\partial \cl{\partial} \log \left[\frac{\partial^2}{\partial z\partial \cl{z}}\phi_{z_0}(z) + N\abs{\Di{z}{\phi_{z_0}}(z)}^2\right] \psi(z)=O(1).
\end{equation}
Note that
\[\log \left[\frac{\partial^2}{\partial z\partial \cl{z}}\phi_{z_0}(z) + N\abs{\Di{z}{\phi_{z_0}}}^2\right]\geq \log \left(\inf_{\C P^1} \frac{\partial^2}{\partial z\partial \cl{z}}\phi_{z_0}(z)\right)=:c.\]
Hence, integrating by parts, we have that (\ref{E:Error II}) is bounded below by $c\cdot \int_{\C P^1}\partial \cl{\partial}\psi.$ We also have that 
 \[\partial \cl{\partial}\log \left[\frac{\partial^2}{\partial z\partial \cl{z}}\phi_{z_0}(z) + N\abs{\Di{z}{\phi_{z_0}}}^2\right] = \partial \cl{\partial}\log \left[\abs{\Di{z}{\phi_{z_0}}}^2+O(N^{-1})\right].\]
Again integrating by parts, we see that (\ref{E:Error II}) is bounded above by $C\cdot \int \partial \cl{\partial}\psi,$ with
\[C:=\sup_{supp(\psi)}\log \abs{\Di{z}{\phi_{z_0}}}.\]
This completes the proof of (\ref{E:Crit Global One Point}).
 
To prove the local asymptotics (\ref{E:Local 1PF Crit}) and (\ref{E:Local 1PF NonCrit}), we fix $\xi \in \C P^1$ and take a $N^{-1/2}-$scale normal coordinate $w$ centered at $\xi.$ As noted in (\ref{E:Raison D'etre}),
\begin{equation}
  \label{E:Scaled Form}
N\cdot \w_h(z)=\frac{i}{2\pi}\partial \cl{\partial} \abs{z}^2 +O(N^{-1/2}).  
\end{equation}
Next, combining the asymptotic expansion (\ref{E:Szego On-Diagonal}) for $\Pi_N(w,w)$ with the lift of $\nabla^{z_0}\otimes \nabla^{z_0}$ to the diagonal of $X\x X$ given in (\ref{E:Lift to Diag at NonCrit}), we see that if $d\phi_{z_0}(\xi)\neq 0$ or $\xi=\infty,$ then 
\[\norm{\nabla_{\Di{z}{}}^{z_0}\otimes \cl{\nabla_{\Di{z}{}}^{z_0}}\Pi_N(w,w)}_{h^N}= N\cdot \left(\abs{\Di{w}{\phi_{z_0}}\bigg|_{\xi}}^2+O(N^{-1/2})\right).\]
Similarly, if $d\phi_{z_0}(\xi)=0,$ then 
\[\norm{\nabla_{\Di{z}{}}^{z_0}\otimes \cl{\nabla_{\Di{z}{}}^{z_0}}\Pi_N(w,w)}_{h^N}= 1+\abs{\DDi{z}{\phi_{z_0}}\bigg|_{\xi}\cdot w+\cl{w}}^2+O(N^{-1/2}).\]
The local asymptotics (\ref{E:Local 1PF Crit}) and (\ref{E:Local 1PF NonCrit}) now follow by substituting the previous two expressions for $\norm{\nabla_V^{z_0}\otimes \cl{\nabla_V^{z_0}} \Pi_N}$ into (\ref{E:Error Gen}) and using (\ref{E:Scaled Form}).
\end{proof}

\section{Conditional Density of Critical Points: Proof of Theorem \ref{T:Crit Given Zero One Point Function}}\label{S:Proof Conditional 1pt}
\noindent We mimic the proofs of Proposition 3.10 and Theorem 1.1 (in the $k=1$ case) in \cite{Conditional}. Let $h$ be a smooth positive Hermitian metric on $\mathcal O(1)\twoheadrightarrow \C P^1.$ Fix $\xi\in \C P^1$ and define $H_N^{\xi}:=\setst{p_N\in H_N}{p_N(z_0)=0}.$ As explained in Section $3$ of \cite{Conditional}, the distribution of $p_N$ conditional on $p_N(\xi)=0$ is the restriction of the Gaussian measure on $H_{hol}^0(\C P^1, \mathcal O(N))$ according to which $p_N$ is distributed to $H_N^{\xi}.$
  
\begin{proof}[Of Theorem \ref{T:Crit Given Zero One Point Function}]
The key to proving the local and global asymptotics is the following result, which is the analog of Proposition 3.10 in \cite{Conditional}.
\begin{Lem}\label{L:Conditional Szego I} For each $N,$
  \begin{equation}
    \label{E:Conditional Szego}
\norm{\nabla_V^{z_0}\otimes \cl{\nabla_V^{z_0}}\Pi_N^{\xi}(w,w)}_{h^N}=\norm{\nabla_V^{z_0}\otimes \cl{\nabla_V^{z_0}}\Pi_N(w,w)}_{h^N}\left(1-\twiddle{P}_N(\xi,w)^2\right),    
  \end{equation}
where
$$\twiddle{P}_N(\xi,w)=\frac{\norm{1\otimes \cl{\nabla_V^{z_0}}\Pi_N(\xi,w)}_h}{\norm{\nabla_V^{z_0}\otimes \cl{\nabla_V^{z_0}}\Pi_N(w,w)}_h^{1/2}\cdot \norm{\Pi_N(\xi,\xi)}_h^{1/2}}$$
is the correlation between $p_N(\xi)$ and $\nabla_V^{z_0}p_N(w)$ (cf Section \ref{S:BS}). 
\begin{proof}
Choose $e_{\xi}^*\in \mathcal O(N)_{\xi}^*$ a unit dual vector. Define the ``coherent state'' at $p_N(\xi)=0$
$$\Psi_N^{\xi}(w):=\frac{1}{\norm{\Pi_N(\xi,\xi)}^{1/2}}\iota_{e_{\xi}^*}\Pi_N(\xi,w)\in H_N,$$
where $\iota$ is the contraction operator. More explicitly, if $\set{S_N^j,~j=0,\ldots,N}$ is any orthonormal basis for $H_N$ with
\[S_N^j(w)=f_j(w)\cdot e_N(w)\]
relative to a local holomorhpic frame $e_N$ for $\mathcal O(N),$ then 
\[\Psi_H^{\xi}(w)=\frac{1}{\left(\sum_{j=0}^N \abs{f_j(\xi)}^2\right)^{1/2}}\sum_{j=0}^N \cl{f_j(\xi)}\cdot S_j(w).\] 
Note that $\inprod{s}{\Psi_N^{\xi}}=0$ for each $s\in H_N^{\xi}.$ Hence, $\Psi_{N}^{z_0}$ spans ${H_N^{z_0}}^{\perp}.$ Further, $\norm{\Psi_N^{\xi}}_h=1.$ So  
$$\norm{\nabla_V^{z_0}\otimes \cl{\nabla_V^{z_0}}\Pi_N^{\xi}(w,w)}_h=\norm{\nabla_V^{z_0}\otimes \cl{\nabla_V^{z_0}}\Pi_N(w,w)}_h-\norm{\nabla_V^{z_0}\Psi_N^{\xi}(w)}_h^2.$$
Noting that
\[\norm{\nabla_V^{z_0} \Psi_N^{\xi}(w)}_h^2=\frac{\norm{1\otimes \cl{\nabla_V^{z_0}}\Pi_N(\xi,w)}_{h^N}^2}{\norm{\Pi_N(\xi,\xi)}_{h^N}}\]
completes the proof. 
\end{proof}
\end{Lem}

We will first prove the local scaling asymptotics (\ref{E:Conditional at Diff Pt})-(\ref{E:Crit Given Zero Scaled II}). Fix $\zeta \in \C P^1$ an a $N^{-1/2}-$scaled normal coordinate centered at $\zeta.$ Combining Corollary \ref{C:Crit Expected Density} and Lemma \ref{L:Conditional Szego I}, we may write 
\begin{equation}
  \label{E:Cond Diff}
\E{C_{p_N}|p_N(\xi)=0}(w)-\E{C_{p_N}}(w)=\frac{i}{2\pi}\partial\cl{\partial}\log (1-\twiddle{P}_N(\xi,w)).  
\end{equation}
If $\zeta\neq \xi,$ then by (\ref{E:Far Off-Diag}) $\twiddle{P}_N(\xi,w)=O(N^{-k})$ for any $k\geq 1,$ confirming (\ref{E:Conditional at Diff Pt}). Next, if $\zeta=\xi,$ then sustituting (\ref{E:Near Off-Diag}) and (\ref{E:Near Off-Diag Infinity}) into (\ref{E:Cond Diff}) proves (\ref{E:Crit Given Zero Scaled I}) and (\ref{E:Crit Given Zero Scaled II}).

To prove the global asymptotics (\ref{E:Global OPF}), we fix $\psi\in C^2(\C P^1)$ and seek to show that
\[\int_{\C P^1} \psi(w) \left[\E{C_{p_N}|p_N(\xi)=0}(w)-\E{C_{p_N}}(w)\right]= O(1).\]
Using (\ref{E:Cond Diff}), this is equivalent to 
\begin{equation}
  \label{E:Conditional Error}
\int_{\C P^1}\psi(w)\frac{i}{2\pi}\partial \cl{\partial}\log (1-\twiddle{P}_N(\xi,w)^2)=O(1).  
\end{equation}
Recall from (\ref{E:Far Off-Diag}) that $\twiddle{P}_N(\xi,w)=O(N^{-k})$ for $\abs{w-\xi}\geq \left(\frac{\log N}{N}\right)^{1/2}.$ Hence, taking a K\"ahler normal coordinate centered at $\xi$, we may write (\ref{E:Conditional Error}) as
\[\int_{\abs{\xi-w}\leq \left(\frac{\log N}{N}\right)^{1/2}}\psi(w)\frac{i}{2\pi}\partial \cl{\partial}\log (1-\twiddle{P}_N(\xi,w)^2)+O(N^{-k})\]
The local asymptotics (\ref{E:Crit Given Zero Scaled I}) and (\ref{E:Crit Given Zero Scaled II}) that we just proved show in particular that the near-diagongal integral is $O(1).$ This concludes the proof. 
\end{proof}

\section{Joint Density of Zeros and Critical Points: Proof of Theorem \ref{T:Crit Two Point Function}}\label{S:Proof 2pt}
\noindent Fix $N\geq 1$ and $h,$ a smooth positive Hermitian metric on $\mathcal O(N)\twoheadrightarrow \C P^1.$ Let $p_N\in H_{hol}^0(\C P^1, \mathcal O(N))$ be drawn from the Hermitian Gaussian ensemble corresponding to $h.$ We seek to compute the local and global asymptotics of the covariance current between the zeros and critical points of $p_N:$
\[\Cov_N(z,w)=\E{Z_{p_N}\wedge C_{p_N}}(z,w)-\E{Z_{p_N}}(z)\wedge \E{C_{p_N}}(w).\]
As in \cite{PLL, NV}, we introduce
\[G(t):=\frac{\gamma^2}{4}-\frac{1}{4}\int_0^{t^2}\frac{\log(1-s)}{s}ds,\quad 0\leq t <1\]
and recall the following 
\begin{Lem}[Lemma 3.3 from \cite{NV}]\label{L:G} Let $a$ be a standard Gaussian random vector in $\C^{N+1}$ and let $u,v\in C^{N+1}$ denote unit vectors. Then 
\[\E{\log\abs{\inprod{a}{u}}\log\abs{\inprod{a}{v}}}=G(\abs{\inprod{u}{v}}),\]
where $\inprod{\cdot}{\cdot}$ is the usual Hermitian inner product on $\C^{N+1}.$
\end{Lem}
\noindent Our first step, Lemma \ref{L:Cov Formula}, gives in the terminology of \cite{NV} a pluri bi-potential for $\Cov_N.$ This result will be the starting point for proving both the local and global asymptotics of $\Cov_N.$
\begin{Lem}\label{L:Cov Formula}
Let $U\subseteq \C P^1.$ In coordinates on $U,$ we have
\begin{equation}
  \label{E:Remainder}
\Cov_N(z,w)=\left(\frac{i}{\pi}\partial_z \cl{\partial}_z\right)\left(\frac{i}{\pi} \partial_w \cl{\partial}_w\right) G(\twiddle{P}_N(z,w)),  
\end{equation}
where $\twiddle{P}_N$ is the absolute value of the correlation kernel between zeros and critical points:
\[\twiddle{P}_N(z,w)=\frac{\norm{1\otimes \cl{\nabla_V^{z_0}} \Pi_N(z,w)}_{h^N}}{\sqrt{\norm{\Pi_N(z,z)}_{h^N}\norm{\nabla_V^{z_0}\otimes\cl{\nabla_V^{z_0}} \Pi_N(w,w)}_{h^N}}}.\]
As elsewhere, $V$ is any auxiliary non-vanishing local holomorphic vector field on which $\twiddle{P}_N$ does not depend. 
\end{Lem}
\begin{proof}
Denote by $e$ a frame for $\mathcal O(1)$ over $U\subseteq \C P^1.$ Recall that $p_N(z)=\sum_{j=0}^N a_j S_j(z)$ for $a_j$ independent standard complex Gaussians and $\set{S_j}$ and orthonormal basis for (\ref{E:Inner Product}). We will write 
\begin{equation}
  \label{E:Abbreviation}
S_j^N = f_j \cdot e^{\otimes N},\quad \nabla_VS_j^N= g_j \cdot e^{\otimes N}  
\end{equation}
and abbreviate $a=[a_0,\ldots, a_N]$ as well as $f=[f_0,\ldots,f_N],\, g=[g_0,\ldots,g_N].$ Denoting by $\inprod{\cdot}{\cdot}$ the usual Hermitian inner product on $\C^{N+1},$ we have 
\[p_N(z)=\inprod{a}{\cl{f}}\cdot e^{\otimes N},\quad \nabla_Vp_N(z)=\inprod{a}{\cl{g}}\cdot e^{\otimes N}.\]
For any smooth test fuction $\psi \in C^{\infty}(\C P\x \C P^1)$ we apply the Poincare-Lelong formula of Lemma \ref{L:PLL} to write $\int_{U\x U} \E{Z_{p_N}\wedge C_{p_N}}(z,w)\psi(z,w)$ as
\[ \E{\int_{U\x U} \frac{i}{\pi}\partial_z \cl{\partial}_z \log \abs{\inprod{a}{\cl{f(z)}}}\,  \frac{i}{\pi}\partial_w \cl{\partial}_w \log \abs{\inprod{a}{\cl{g(w)}}}\psi(z,w)}.\]
We now integrate by parts and, using that the resulting integrand is in $L^1,$ exchange the order of integration to find
\begin{equation}
  \label{E:K2 PLL}
\E{Z_{p_N}\wedge C_{p_N}}(z,w) =\left(\frac{i}{\pi}\partial_z \cl{\partial}_z \wedge \frac{i}{\pi}\partial_w \cl{\partial}_w\right)\E{\log \abs{\inprod{a}{f(z)}}\log\abs{\inprod{a}{g(w)}}}.
\end{equation}
Just as in Section 3 of \cite{PLL}, we introduce $u(z):=\frac{\cl{f(z)}}{\norm{f(z)}},\, v(z):=\frac{\cl{g(w)}}{\norm{g(w)}},$ to write
\begin{align*}
\E{C_{p_N}\wedge Z_{p_N}}(z,w)&=-\frac{1}{\pi^2}\left(E_1(z,w)+E_2(z,w)+E_3(z,w)+E_4(z,w)\right),
\end{align*}
where
\begin{align*}
E_1(z,w)&:= \partial_z\cl{\partial}_z\partial_w\cl{\partial}_w \log \norm{f(z)}\log\norm{g(w)}\\
E_2(z,w)&:= \partial_z\cl{\partial}_z\partial_w\cl{\partial}_w \E{\log \abs{\inprod{a}{u(z)}} \cdot\log\norm{g(w)}}\\
E_3(z,w)&:= \partial_z\cl{\partial}_z\partial_w\cl{\partial}_w \E{\log \abs{\inprod{a}{v(w)}} \cdot\log\norm{f(z)}}\\
E_4(z,w)&:= \partial_z\cl{\partial}_z\partial_w\cl{\partial}_w \E{\log \abs{\inprod{a}{u(z)}} \log \abs{\inprod{a}{v(w)}}}.
\end{align*}
As in the proof of Theorem \ref{T:EZ}, $E_2$ and $E_3$ vanish as they are independent of $z,w,$ respectively (see also Section 3.2 of \cite{Quantum}). Further, as $\norm{f(z)}=\sum_{j=0}^N \abs{f_j(z)}^2$ and $\norm{g(w)}=\sum_{j=0}^N \abs{\Di{w}{f_j}(w)}^2,$ we see from Corollary \ref{C:Crit Expected Density} and Lemma \ref{L:Zero Expected Density} that 
\begin{equation}
  \label{E:E1 Term}
-\frac{1}{\pi^2}E_1(z,w)=\E{Z_{p_N}}(z)\wedge \E{C_{p_N}}(w).  
\end{equation}
Hence, $-\frac{1}{\pi^2}E_4(z,w)=\Cov_N(z,w)$ and by Lemma \ref{L:G} 
\begin{equation}
  \label{E:E2 Term}
E_4(z,w)=\partial_z\cl{\partial}_z\partial_w\cl{\partial}_w G(\abs{\inprod{u(z)}{v(w)}}).  
\end{equation}
Observing that $\abs{\inprod{u(z)}{v(w)}}=\twiddle{P}_N(z,w)$ completes the proof. 
\end{proof}
The local asymptotics (\ref{E:Cov BiPot}) follow immediately by substituting into equation (\ref{E:Remainder}) the asymptotic expansions (\ref{E:Near Off-Diag}) and (\ref{E:Near Off-Diag Infinity}). To prove the global asymptotics (\ref{E:Cov BiPot}), it suffices by Lemma \ref{L:Cov Formula} to show that for any test function $\psi \in C^{\infty}(\C P^1\x \C P^1)$ and any $\ep>0$ that
\begin{equation}
  \label{E:Global Remainder}
\int_{\C P^1\x \C P^1} G(\twiddle{P}_N(z,w))\partial_z \cl{\partial}_z \partial_w \cl{\partial}_w\psi(z,w)= O(N^{\ep}).  
\end{equation}
Recall from (\ref{E:Normalized Estimate Compat}) and (\ref{E:Normalized Estimate Non-Compat}) that $\twiddle{P}_N(z,w)<1$ for each fixed $N.$ Since $G(t)$ is a smooth away from $t=1,$ we see that $G(\twiddle{P}_N(z,w))$ is smooth. Moreover, $G(t)$ is stictly increasing in $t$ and the singularity at $t=1$ is given by $G(1-t)\sim \log(t).$ Thus, since the weakest estimate (\ref{E:Normalized Estimate Non-Compat}) holds for all $z,w$ we see that
\[G(\twiddle{P}_N(z,w))= O\left(\log N\right)=O(N^{\ep}),\]
allowing us to conclude (\ref{E:Global Remainder}). 

\section{Expected Nearest Neighbor Spacings: Proof of Theorem \ref{T:Sendov}}\label{S:Proof Sendov}
\noindent Fix $N\geq 1,$ $\ep>0,$ and $\xi \in \C P^1.$ In $N^{-1/2}-$scale normal coordinates centerd at $\xi,$ we fix $A\subseteq \C$ measurable with finite area. Our goal is to estimate 
\begin{equation}
  \label{E:Expected Number NonCrit}
\E{\curly X_{N,A,\ep}}=\int_{\C^2}I_{w\in A}(w)\cdot I_{\abs{z-w}\leq N^{-1/4+\ep}}(z,w) \E{Z_{p_N}\wedge C_{p_N}}(z,w)  
\end{equation}
when $d\phi_{z_0}(\xi)\neq 0$ or $\xi=\infty$ and
\begin{equation}
  \label{E:Expected Number Crit}
\int_{\C^2}I_{w\in A}(w) \cdot I_{\abs{z-w}\in [\abs{\zeta}^{-1}\pm \abs{\zeta}^{-1-c}],\, \, \arg(z-w)\in [\arg(\zeta)\pm \abs{\zeta}^{-c}]}(z,w) \E{Z_{p_N}\wedge C_{p_N}}(z,w)  
\end{equation}
when $d\phi_{z_0}(\xi)=0.$ By definition of $\Cov_N,$
\begin{align*}
\E{Z_{p_N}\wedge C_{p_N}}(z,w)&= \Cov_N(z,w)+\E{Z_{p_N}}(z)\wedge \E{C_{p_N}}(w).
\end{align*}
From Theorems  \ref{T:Crit Two Point Function} and \ref{T:EZ}, recall that 
\[\E{Z_{p_N}}(z)=\frac{i}{2\pi}\partial_z\cl{\partial}_z \abs{z}^2 +O(N^{-1/2+\ep})\]
while
\[\E{C_{p_N}}(w)=\frac{i}{2\pi}\partial_w\cl{\partial}_w \abs{w}^2+O(N^{-1/2+\ep})\]
when $d\phi_{z_0}(\xi)\neq 0$ or $\xi=\infty$ and, writing $\zeta=\DDi{\cl{w}}{\phi_{z_0}}\big|_{\xi}\cl{w}+w,$ 
\[\E{C_{p_N}}(w)=\frac{i}{2\pi}\partial_w\cl{\partial}_w \abs{w}^2+\frac{i}{2\pi}\partial_w\cl{\partial}_w \log(1+\abs{\zeta}^2)+O(N^{-1/2+\ep})\]
if $d\phi_{z_0}(\xi)=0.$ Hence, the contribution of $\E{Z_{p_N}}\wedge \E{C_{p_N}}$ to the integral (\ref{E:Expected Number NonCrit}) is $O(N^{-1/4}).$ Similarly, the contibution to (\ref{E:Expected Number Crit}) is $O(N^{-1/2+\ep})+O\left(\int_A \abs{\zeta}^{-2}\right).$
To prove (\ref{E:Sendov NonCrit}) and (\ref{E:Crit Sendov Pairing}), we therefore focus on estimating the integrals (\ref{E:Expected Number NonCrit}) and (\ref{E:Expected Number Crit}) with $\E{C_{p_N}\wedge Z_{p_N}}$ replaced by $\Cov_N(z,w).$ We do this by using the bi-potential obtained in Lemma \ref{L:Cov Formula}: 
\[\Cov_N(z,w)=\left(\frac{i}{\pi}\partial_z\cl{\partial}_z\wedge \frac{i}{\pi}\partial_w\cl{\partial}_w\right) G(\twiddle{P}_N(z,w)).\]
\subsection{Case 1 ($d\phi_{z_0}(\xi)\neq 0$)}
We begin with the case $d\phi_{z_0}(\xi)\neq 0$ or $\xi=\infty.$ 
\begin{Lem}
For each $N$ and $A$ and for some $R>0,$ 
\begin{equation}
  \label{E:Counting Number}
\E{\curly X_{N,A}}=\frac{1}{\pi^2}\int_{z\in A}\int_{\abs{z-w}\leq N^{-1/4}} \partial_w\cl{\partial}_w \left[\log(\abs{z-w}^2+R\cdot N^{-1/2})+O(1).\right]  
\end{equation} 
\end{Lem}
\begin{proof}
Recall that $G(t)=\frac{\gamma^2}{2}-\frac{1}{4}\int_0^{t^2}\frac{\log(1-s)}{s}ds$ for $0\leq t <1.$ Thus,
\begin{align*}
\partial_z\cl{\partial}_z G(\twiddle{P}_N(z,w)) &= \partial_z \left[\frac{\log(1-\twiddle{P}_N(z,w)^2)}{\twiddle{P}_N(z,w)^2}\cdot \cl{\partial}_z(\twiddle{P}_N(z,w)^2)\right]
\end{align*}
Using (\ref{E:P Est 1})-(\ref{E:P Est 3}), we may rewrite this as
\begin{align*}
\partial_z\cl{\partial}_z G(\twiddle{P}_N(z,w))&= \frac{1}{\pi^2}\partial_z \left[\log(1-\twiddle{P}_N(z,w)^2)\cdot \cl{\partial}_z(\twiddle{P}_N(z,w)^2)\right]\\
&= \frac{1}{\pi^2}\log(1-\twiddle{P}_N(z,w)^2)\cdot \left(\partial_z\cl{\partial}_z\right)(\twiddle{P}_N(z,w)^2)+ \frac{\abs{\partial_z \twiddle{P}_N(z,w)^2}^2}{1-\twiddle{P}_N(z,w)^2}\\
&=\frac{1}{\pi^2}\log(\abs{z-w}^2+R\cdot N^{-1/2} +O(N^{-3/4+\ep}))+O(1).
\end{align*}
\end{proof}
\noindent To understand (\ref{E:Counting Number}), we prove the following perturbation of Laplace's law: $\Delta \log \abs{z}^2=\pi\cdot \delta_0.$
\begin{Lem}[Perturbed Laplace Equation]\label{L:Perturbed Laplace}
  Fix $\alpha>0$ and $\psi \in C_c^0(\C).$ Then for any $\ep>0$ and constant $K>0,$
  \begin{equation}
    \label{E:Laplace's Law}
    \int_{\abs{u}\leq \frac{\log N}{N^{\alpha}}} \left[\frac{i}{2}\partial_u\cl{\partial}_u \log(\abs{u}^2+K\cdot N^{-\alpha})\right]\psi(u)=\pi \psi(0)+O(N^{-\alpha+\ep}).
  \end{equation}
\end{Lem}
\begin{proof}
  The proof is by direct computation. We have
  \begin{align*}
    \frac{i}{2}\partial_u\cl{\partial}_u \log(\abs{u}^2+K\cdot N^{-\alpha}) &= \frac{K\cdot N^{-\alpha}}{(\abs{u}^2+K\cdot N^{-\alpha})^2}\frac{i}{2}du\wedge d\cl{u}.
  \end{align*}
Hence, denoting by $I$ the left hand side of (\ref{E:Laplace's Law}) and making the change of coordinates $u\mapsto u\cdot K^{1/2}\cdot N^{-\alpha/2},$ we have
\begin{align*}
  I &= \int_{\abs{u}\leq \log N}\frac{1}{(1+\abs{u}^2)^2}\psi\left(\frac{u}{N^{\alpha}}\right)\,\,\frac{i}{2}du\wedge d\cl{u}\\
    &= \pi\psi(0)+O(N^{-\alpha+\ep}),
\end{align*}
by the continuity of $\psi$ and the fact that $\frac{i}{2\pi}\frac{1}{(1+\abs{u}^2)^2}du\wedge d\cl{u}$ is the volume $1$ Fubini-Study measure.
\end{proof}

\noindent For each fixed $z,$ we make the change of variables $u=z-w$ in (\ref{E:Counting Number}) and apply Lemma \ref{L:Perturbed Laplace}. Recalling that $\E{C_A}=\int_A \frac{i}{2\pi}dw\wedge d\cl{w}+O(N^{-1/2})$ completes the proof of Theorem \ref{T:Sendov} in this case. 

\subsection{Case 2 ($d\phi_{z_0}(\xi)=0$)}\label{S:Explicit Formula}
We fix $\xi \in \C P^1\backslash \set{\infty}$ such that $d\phi_{z_0}(\xi)=0$ and take a $N^{-1/2}-$scale normal coordinate for $\w_h$ centered at $\xi.$ We then fix a bounded measurable set $A\subseteq \C\backslash \set{\abs{\zeta}\leq 1}$ and a parameter $c\in(2/3,1).$ We've denoted $\alpha=\DDi{\cl{w}}{\phi_{z_0}}(\xi)$ at in the statement of the Theorem, and we define
\[\zeta:=\alpha \cl{w}+w,\quad \eta:=\alpha \cl{w}+z.\]
The proof in this case is a long computation. For the reader's convenience we provide an outline. 

\subsubsection*{Outline} \label{S:Outline of Computation}
Recall that $G(t)=\frac{\gamma^2}{4}-\frac{1}{4}\int_0^{t^2}
\frac{\log(1-s)}{s}ds.$ As explained in the beginning of the proof, we seek to compute
\[\int_{\C\x\C}I_{w\in A}(w) \cdot I_{\abs{z-w}\in [\abs{\zeta}^{-1}\pm \abs{\zeta}^{-1-c}],\, \, \arg(z-w)\in [\arg(\zeta)\pm \abs{\zeta}^{-c}]}(z,w) \Cov_N(z,w).\]
Our first step is to recall from Lemma \ref{L:Cov Formula} that
\[\Cov_N(z,w)=-\frac{1}{\pi^2}\partial_z\cl{\partial}_z\partial_w\cl{\partial}_w G(\twiddle{P}_N(z,w)).\]
Writing $\frac{i}{2}dw \wedge d\cl{w}$ for the usual lebesgue volume form we seek to evaluate
\begin{equation}
  \label{E:Counting Int}
\frac{4}{\pi}\cdot \int_A\left[\int_{\arg(\zeta)-\abs{\zeta}^{-c}}^{\arg(\zeta)+\abs{\zeta}^{-c}}\int_{\abs{\zeta}^{-1}-\abs{\zeta}^{-1-c}}^{\abs{\zeta}^{-1}+\abs{\zeta}^{-1-c}} \frac{\partial^2}{\partial z \partial \cl{z}}\frac{\partial^2}{\partial w \partial \cl{w}}G(\twiddle{P}_N(w+re^{it},w)\, rdrdt\right]\,\, \frac{i}{2\pi}dw\wedge d\cl{w} . 
\end{equation}
\noindent Our second step is to use (\ref{E:Near Off-Diag}) to write
\[\twiddle{P}_N(z,w)=\twiddle{P}_{\infty}(z,w)+O(N^{-1/2+\ep}),\]
where
\[\twiddle{P}_{\infty}(z,w)=\frac{\abs{\eta}}{\sqrt{1+\abs{\zeta}^2}}e^{-\frac{1}{2}\abs{\eta-\zeta}^2},
\]
is the $N\gives \infty$ limit of $\twiddle{P}_N(z,w).$ We have already seen in Corollary \ref{C:Normalized Estimates} that, for each $N,$ $\twiddle{P}_N(z,w)\leq C <1$ for some universal constant $C.$ Hence, since $G(t)$ is smooth away from $t=1,$ the integrand in (\ref{E:Counting Int}) becomes
\[\frac{\partial^2}{\partial z \partial \cl{z}}\frac{\partial^2}{\partial w \partial \cl{w}}G(\twiddle{P}_N(z,w))=\frac{\partial^2}{\partial z \partial \cl{z}}\frac{\partial^2}{\partial w \partial \cl{w}}G(\twiddle{P}_{\infty}(z,w))+O(N^{-1/2+\ep}).\]
Our third step is to compute the four derivatives of $G.$ It is necessary to do so since we seek to compute non-smooth statistics and hence we cannot integrate by parts onto the test function. To this end, we introduce as in Lemma 3.5 of \cite{NV} 
\[\leb(z,w):=-\log \twiddle{P}_{\infty}(z,w)=\frac{1}{2}\left[\abs{z-w}^2+\log(1+\abs{\alpha \cl{w}+w}^2)-\log(\abs{\alpha \cl{w}+z}^2)\right].\]
We further write for $\leb \geq 0$
\[F(\leb):=G(e^{-\leb})=\frac{\gamma^2}{4}-\frac{1}{2}\int_{\leb}^{\infty} \log(1-e^{-2s})ds.\]
We will compute 
\[\frac{\partial^2}{\partial z \partial \cl{z}}\frac{\partial^2}{\partial w \partial \cl{w}}F(\leb(z,w)).\]
It will turn out (cf (\ref{E:4 Derivs})) that these four derivatives may be written as
\[F^4(\leb(z,w))\cdot \abs{\Di{z}{}\leb(z,w)}^2\cdot \abs{\Di{w}{}\leb(z,w)}^2+F^3(\leb(z,w))\cdot \abs{\Di{z}{}\leb(z,w)-\Di{w}{}\leb(z,w)}^2+\frac{1}{2}F^2(\leb(z,w))\]
plus a small error, where we've written $F^j$ for the $j^{th}$ derivative of $F.$ In Lemmas \ref{L:Crit Dist} and \ref{L:Partials Estimates} we obtain estimates on the derivatives of $F$ and $\leb$ in the regime of the integration in (\ref{E:Counting Int}). Finally, we combine these estimate and read off the desired answer.  
\subsubsection*{Computations}
We write $F^j$ for the $j^{th}$ derivative of $F$ and abbreviate
$$\partial_z := \Di{z}{}\leb(z,w),~~\partial_w:= \Di{w}{} \leb(z,w),~~(\partial_z \partial_{\cl{z}}):=\frac{\partial^2}{\partial z \partial \cl{z}}\leb(z,w)$$ 
and so on. We will need the following derivatives of $\leb(z,w):$ 
\begin{align}
\label{E:Derivs I}\partial_z &=\frac{1}{2}\left(\cl{\eta}-\cl{\zeta}-\frac{1}{\eta}\right)~~~&(\partial_z\partial_{\cl{z}}) &=\frac{1}{2}\\
\label{E:Derivs II}\partial_w &=\frac{1}{2}\left(\frac{\cl{\alpha}\zeta+\cl{\zeta}}{1+\abs{\zeta}^2}-\cl{\eta}+\cl{\zeta}-\frac{\cl{\alpha}}{\cl{\eta}}\right)~~~&(\partial_w\partial_{\cl{w}}) &=\frac{1}{2}\left(1+O(\abs{\zeta}^{-2})\right)\\
\label{E:Derivs III}(\partial_z\partial_{\cl{w}})&=(\partial_w\partial_{\cl{z}}) =-\frac{1}{2}\left(1+O(\abs{\zeta}^{-2})\right),~~~&(\partial_z\partial_w)&=(\partial_{\cl{z}}\partial_{\cl{w}})=0.
\end{align}
We have supressed all the delta functions $\delta_{\eta}$ since we are interested in these derivatives when $\abs{\zeta}>1$ and $\abs{\zeta-\eta}$ is small. Using these explicit formuae, we have 
\begin{align}\label{E:4 Derivs}
\frac{\partial^2}{\partial z \partial \cl{z}}\frac{\partial^2}{\partial w \partial \cl{w}}(G(\twiddle{P}_{\infty})) 
&=F^4 \cdot \abs{\partial_z}^2 \cdot \abs{\partial_w}^2+F^3\cdot \frac{1}{2}\left[\abs{\partial_w-\partial_z}^2+O(\abs{\zeta}^{-2})\right]+ F^2\cdot \frac{1}{4}\left[2+O(\abs{\zeta}^{-2})\right].
\end{align} 
Next, we record the derivatives of $F$ evaluated at $\leb>0:$
\begin{align}
\label{E:F Derivs 1}F^1(\leb)&=\frac{1}{2}\log(1-e^{-2\leb}),~~~&F^2(\leb)&=\frac{e^{-2\leb}}{1-e^{-2\leb}}\\
\label{E:F Derivs 2}F^3(\leb)&=-2\cdot \frac{e^{-2\leb}}{(1-e^{-2\leb})^2},~~~&F^4(\leb)&=4\cdot \frac{e^{-2\leb}+e^{-4\leb}}{(1-e^{-2\leb})^3}.
\end{align}
\noindent Next, in order to estimate the derivatives of $F$ and $\leb$ in the region of interest for the integral (\ref{E:Counting Int}), we write 
\[z=w+re^{it}\quad \text{and}\quad r=\frac{1+\ep}{\abs{\zeta}}\] 
and assume that 
\begin{equation}
  \label{E:Ep and t Constraints}
\abs{\zeta}>1,\quad\text{and}\quad \ep \in [-\abs{\zeta}^{-c},\,\abs{\zeta}^{-c}],\quad \text{and}\quad t\in [\arg(\zeta) -\abs{\zeta}^{-c},\,\arg(\zeta)+\abs{\zeta}^{-c}].
\end{equation}
Observe that for this range of $z,w$ we have $\twiddle{P}_{\infty}(z,w)=1+O(\abs{\zeta}^{-2}).$ Hence, the numerator in the above expression (\ref{E:F Derivs 2}) for $F^4$ may be written
\[4(e^{-4\leb}+e^{-2\leb})=8\twiddle{P}_{\infty}(z,w)^2(1+(\abs{\zeta}^{-2})).\]
The integrand in the integral (\ref{E:Counting Int}) we seek to compute may be written as 
\begin{equation}
  \label{E:Counting Int III}
\frac{4}{\pi^2}\cdot \left[\frac{\frac{1}{2}(1+O(\abs{\zeta}^{-2}))rdr\wedge dt}{1-\twiddle{P}_{\infty}(z,w)^2}\right]\left(\frac{16\cdot \abs{\partial_z}^2 \cdot \abs{\partial_w}^2}{(1-\twiddle{P}_{\infty}(z,w)^2)^2}-\frac{2\cdot \abs{\partial_w-\partial_z}^2}{1-\twiddle{P}_{\infty}(z,w)^2}+ 1\right)\,\, \frac{i}{2}dw\wedge d\cl{w}  
\end{equation}
\noindent To evaluate (\ref{E:Counting Int III}), we give estimates on $(1-\twiddle{P}_{\infty}(z,w)^2)^{-1}$ in Lemma \ref{L:Crit Dist} and on the various derivatives of $\leb$ in Lemma \ref{L:Partials Estimates}.  
\begin{Lem}\label{L:Crit Dist}
Write $r=\frac{1+\ep}{\abs{\zeta}}$ and assume (\ref{E:Ep and t Constraints}). We have the following expression for $(1-\twiddle{P}_{\infty}(w+re^{it},w)^2)^{-1}:$
\begin{equation}
  \label{E:Normalized Szego Estimate}
\frac{\abs{\zeta}^4(1+O(\abs{\zeta}^{-c}))}{\frac{1}{2}+(\abs{\zeta} \ep)^2+(\abs{\zeta}(t-\arg(\zeta)))^2}.
\end{equation}
\end{Lem}
\begin{proof}
Recall that 
\[\twiddle{P}_{\infty}(z,w)^2=\frac{\abs{\eta}^2}{1+\abs{\zeta}^2}e^{-\abs{\eta-\zeta}^2}.\]
This is invariant under rotation of $(\eta,\zeta).$ Hence, we may assume that $\zeta\in \mathbb R_+$ and that $t\in [-\abs{\zeta}^{-c},\,\abs{\zeta}^{-c}].$ We write
\[(1-\twiddle{P}_{\infty}(z,w)^2)^{-1}=\frac{(1+\zeta^2)e^{\abs{\eta-\zeta}^2}}{(1+\zeta^2)e^{\abs{\eta-\zeta}^2}-\abs{\eta}^2}.\]
Using that $\abs{\eta}^2=\zeta^2+r^2+2(1+\ep)\cos t$ and taylor expanding $e^{\abs{\eta-\zeta}^2}$ and $\cos t$ we obtain (\ref{E:Normalized Szego Estimate}).
\end{proof}
\begin{Lem}\label{L:Partials Estimates}
Fix $w.$ Write $z=w+re^{it}$ and assume the regime (\ref{E:Ep and t Constraints}). We have
\begin{equation}
  \label{E:Estimates}
\abs{\partial_z\leb(z,w)}^2=\abs{\partial_w\leb(z,w)}^2=\frac{1}{4}\abs{\partial_z\leb(z,w)-\partial_w\leb(z,w)}^2=\frac{\ep^2+(t-\arg(\zeta))^2}{4\abs{\zeta}^2} +O(\abs{\zeta}^{-2-3c}).
\end{equation}
\end{Lem}
\begin{proof}
Observe that $\frac{1}{\eta}-\frac{1}{\zeta}=O(\abs{\zeta}^{-3}).$ Hence, from (\ref{E:Derivs I}),
\[\abs{\partial_z\leb(z,w)}^2=\frac{1}{4}\abs{\cl{\eta}-\cl{\zeta}-\frac{1}{\zeta}+O(\abs{\zeta}^{-3})}^2.\]
This expression is invariant under rotation of $(\eta,\zeta).$ Hence, we may assume that $\zeta\in \mathbb R_+.$ The equation $\eta=\zeta+re^{it}$ then reads 
\begin{equation}
  \label{E:Eta Zeta Real}
\eta=\zeta+\frac{1+\ep}{\zeta}e^{it}.  
\end{equation}
Using that 
\begin{equation}
  \label{E:Final Ep}
\abs{(1+\ep)-e^{-it}}^2=\ep^2+2(1+\ep)(1-\cos t) =\ep^2+t^2+O(\zeta^{-3c})
\end{equation}
allows us to conclude the desired expression for $\abs{\partial_z\leb(z,w)}^2.$ To prove the same expression for $\abs{\partial_w\leb(z,w)}^2,$ we observe that
\begin{equation}
  \label{E: LebW Est}
\frac{\cl{\alpha}\zeta}{1+\abs{\zeta}^2}-\frac{\cl{\alpha}}{\cl{\eta}}=O(\abs{\zeta}^{-3}).  
\end{equation}
We then write
\[\abs{\partial_w\leb(z,w)}^2=\frac{1}{4}\abs{\frac{\zeta}{1+\abs{\zeta}^2}+\zeta-\eta+O(\abs{\zeta}^{-3})}^2.\]
As before, this expression is under rotation of $(\eta,\zeta)$ so we may take $\zeta\in \mathbb R_+.$ We then use that 
\[\frac{\zeta}{1+\zeta^2}-\frac{1}{\zeta}=O(\zeta^{-3})\]
and equations (\ref{E:Eta Zeta Real})-(\ref{E:Final Ep}) to conclude the desired expression for $\abs{\partial_w\leb(z,w)}^2.$ Finally, 
\[\abs{\partial_z\leb(z,w)-\partial_w\leb(z,w)}^2=\abs{2(\cl{\eta}-\cl{\zeta}-\frac{1}{\cl{\eta}})-\frac{\zeta}{1+\abs{\zeta}^2}+\frac{1}{\cl{\eta}}+O(\abs{\zeta}^{-3})}.\]
This expression is invariant under rotation of $(\eta,\zeta).$ Taking $\zeta\in \mathbb R_+,$ we see from (\ref{E:Eta Zeta Real}) that $\cl{\eta}-\cl{\zeta}-\frac{1}{\cl{\zeta}}=(1+\ep)e^{-it}-1.$ Observing $-\frac{\zeta}{1+\abs{\zeta}^2}+\frac{1}{\cl{\eta}}=O(\abs{\zeta}^{-3})$ and using (\ref{E:Final Ep}) concludes the proof.
\end{proof}

\noindent We define $D(r,t):=\frac{\ep^2+(t-\arg(\zeta))^2}{\abs{\zeta}^2}.$ Using Lemma \ref{L:Partials Estimates}, we may write the integrand in the formula (\ref{E:Counting Int III}) for $\E{\curly X_{A,N}}$ as
\begin{equation}
  \label{E:Counting Int II}
\frac{\frac{1}{2}(1+O(\abs{\zeta}^{-2}))r dr\wedge dt}{1-\twiddle{P}_{\infty}(w+re^{it},w)^2}\left(1-\frac{D(r,t)}{1-\twiddle{P}_{\infty}(w+re^{it},w)^2}\right)^2.  
\end{equation}
\noindent We now change variables in (\ref{E:Counting Int II}) from $(r,t)$ to $(x:=\abs{\zeta}\cdot \ep,\,\, y:=\abs{\zeta}\cdot (t-\arg(\zeta))).$ From the estimate (\ref{E:Normalized Szego Estimate}), we have that
\begin{equation}
  \label{E:First Term}
\frac{\frac{1}{2}(1+O(\abs{\zeta}^{-2}))rdr\wedge dt}{1-\twiddle{P}_{\infty}(z,w)^2}=\frac{\frac{1}{2}\left(1+O(\abs{\zeta}^{-1-c})\right)dx \wedge dy}{\frac{1}{2}+x^2+y^2}.
\end{equation}
Similarly, combining (\ref{E:Normalized Szego Estimate}) with (\ref{E:Estimates}), we find that 
\begin{equation}
  \label{E:Second Term}
\left(1-\frac{D(r,t)}{1-\twiddle{P}_{\infty}(w+re^{it},w)^2}\right)^2=\frac{1}{4}\frac{1}{(\frac{1}{2}+x^2+y^2)^2}.  
\end{equation}
Combining (\ref{E:Normalized Szego Estimate}) and (\ref{E:Estimates}), we write the inner integral in (\ref{E:Counting Int}) as
\[\frac{1}{8}\int_{-\abs{\zeta}^{1-c}}^{\abs{\zeta}^{1-c}}\int_{-\abs{\zeta}^{1-c}}^{\abs{\zeta}^{1-c}}\frac{(1+O(\abs{\zeta}^{-c}) +O(N^{-1/2+\ep}))dx\,dy}{(1/2+x^2+y^2)^3}\]
passing to polar coordinates we see that for $c<1,$ this integral equals $\frac{\pi}{4}+O(\abs{\zeta}^{2-3c})+O(N^{-1/2+\ep}).$ Substituting this into (\ref{E:Counting Int}) and recalling from (\ref{E:Local 1PF Crit}) of Theorem \ref{T:EZ} that $\E{C_A}=\int_A \frac{i}{2\pi}dw\wedge d\cl{w}+O\left(\int_A \abs{\zeta}^{-2}\right)+O(N^{-1/2+\ep})$ completes the proof. 
\bibliographystyle{plain}
\bibliography{Bibliography}

\end{document}